\documentclass{amsart}

\usepackage{amsmath}
\usepackage{amsthm}
\usepackage{hyperref}
\usepackage{amsfonts,graphics,amsthm,amsfonts,amscd,latexsym}
\usepackage{epsfig}
\usepackage{flafter}
\usepackage{mathtools}
\usepackage{comment}
\usepackage{stmaryrd}

\usepackage{mathabx,epsfig}

\hypersetup{
    colorlinks=true,    
    linkcolor=blue,          
    citecolor=blue,      
    filecolor=blue,      
    urlcolor=blue           
}
\usepackage{tikz}
\usetikzlibrary{graphs,positioning,arrows,shapes.misc,decorations.pathmorphing}

\tikzset{
    >=stealth,
    every picture/.style={thick},
    graphs/every graph/.style={empty nodes},
}

\tikzstyle{vertex}=[
    draw,
    circle,
    fill=black,
    inner sep=1pt,
    minimum width=5pt,
]
\usepackage[position=top]{subfig}
\usepackage{amssymb}
\usepackage{color}

\calclayout

\usetikzlibrary{decorations.pathmorphing}
\tikzstyle{printersafe}=[decoration={snake,amplitude=0pt}]

\newcommand{\Cl}{\operatorname{Cl}}

\newcommand{\loc}{\operatorname{loc}}

\newcommand{\supp}{\operatorname{supp}}

\newcommand{\Spec}{\operatorname{Spec}}

\newcommand{\Hom}{\operatorname{Hom}}

\newcommand{\relint}{\operatorname{relint}}

\newcommand{\WDiv}{\operatorname{WDiv}}

\newcommand{\PDiv}{\operatorname{PDiv}}

\newcommand{\CaDiv}{\operatorname{CaDiv}}

\newcommand{\ddivv}{\operatorname{div}}

\newcommand{\pol}{\operatorname{Pol}}

\newcommand{\ver}{\operatorname{Ver}}
\newcommand{\ray}{\operatorname{Ray}}

\newcommand{\Aut}{\operatorname{Aut}}

\newcommand{\qq}{\mathbb{Q}}
\newcommand{\zz}{\mathbb{Z}}
\newcommand{\nn}{\mathbb{N}}

\newcommand{\cc}{\mathbb{C}}
\newcommand{\kk}{\mathbb{K}}

\newcommand{\D}{\mathcal{D}}
\newcommand{\G}{\mathbb{G}}
\newcommand{\TT}{\mathbb{T}}
\newcommand{\oo}{\mathcal{O}}
\newcommand{\A}{\mathcal{A}}

\def\O#1.{\mathcal {O}_{#1}}			
\def\pr #1.{\mathbb P^{#1}}				
\def\af #1.{\mathbb A^{#1}}			
\def\ses#1.#2.#3.{0\to #1\to #2\to #3 \to 0}	
\def\xrar#1.{\xrightarrow{#1}}			
\def\K#1.{K_{#1}}						
\def\bA#1.{\mathbf{A}_{#1}}			
\def\bM#1.{\mathbf{M}_{#1}}				
\def\bL#1.{\mathbf{L}_{#1}}				
\def\bB#1.{\mathbf{B}_{#1}}				
\def\bK#1.{\mathbf{K}_{#1}}			
\def\subs#1.{_{#1}}					
\def\sups#1.{^{#1}}

\usepackage{tikz}
\usetikzlibrary{matrix,arrows,decorations.pathmorphing}

\newtheorem{introthm}{Theorem}

\newtheorem{introcor}{Corollary}

  \newtheorem{introquest}{Question}

  \newtheorem{theorem}{Theorem}[section]
  \newtheorem{lemma}[theorem]{Lemma}
  \newtheorem{proposition}[theorem]{Proposition}

  \newtheorem{definition}[theorem]{Definition}
  \newtheorem{example}[theorem]{Example}

  \newtheorem{construction}[theorem]{Construction}

\newtheorem{remark}[theorem]{Remark}

\theoremstyle{remark}

\numberwithin{equation}{section}

\usepackage[all]{xy}

\begin{document}

\title[Reductive quotients of klt singularities]{Reductive quotients of klt singularities}

\author[L.~Braun]{Lukas Braun}
\address{Institut für Mathematik, Fakultät für Mathematik, Informatik und Physik, Universität Innsbruck, Technikerstraße 13, A-6020 Innsbruck, Austria}
\email{lukas.braun@uibk.ac.at}

\author[D. Greb]{Daniel Greb}
\address{Essener Seminar f\"ur Algebraische Geometrie und Arithmetik, Fakult\"at f\"ur Mathematik, Universität Duisburg-Essen,  45117 Essen}
\email{daniel.greb@uni-due.de}

\author[K.~Langlois]{Kevin Langlois}
\address{Departemento de Matem\'atica, Universidade Federal do Cear\'a (UFC), Campus do Pici, Bloco 914, CEP 60455-760. Fortaleza-Ce, Brazil}
\email{kevin.langlois@mat.ufc.br }

\author[J.~Moraga]{Joaqu\'in Moraga}
\address{Department of Mathematics, Princeton University, Fine Hall, Washington Road, Princeton, NJ 08544-1000, USA
}
\email{jmoraga@princeton.edu}

\thanks{
During the work on this article, LB was partially supported by DFG grant 452847893, FWF grant PAT9495823, and by the DFG Research Training Group GRK 1821. DG is partially supported by the DFG Research Training Group GRK 2553 ``Symmetries and Classifying Spaces: analytic, arithmetic and derived". 
}


\begin{abstract}
We prove that the quotient of a klt type singularity by a reductive group is of klt type in characteristic $0$. In particular, given a klt variety $X$ endowed with the action of a reductive group $G$ and admitting a quasi-projective good quotient $X\rightarrow X/\!/G$, we can find a boundary $B$ on $X/\!/G$ so that the pair $(X/\!/G,B)$ is klt. This applies for example to GIT-quotients of klt varieties. Our main result has consequences for complex spaces obtained as quotients of Hamiltonian K\"ahler $G$-manifolds, for collapsings of homogeneous vector bundles as introduced by Kempf, and for good moduli spaces of smooth Artin stacks.
In particular, it implies that the good moduli space parametrizing $n$-dimensional K-polystable smooth Fano varieties of volume $v$ has klt type singularities. As a corresponding result regarding global geometry, we show that quotients of Mori Dream Spaces with klt Cox rings are Mori Dream Spaces with klt Cox ring. This in turn applies to show that projective GIT-quotients of varieties of Fano type are of Fano type; in particular, projective moduli spaces of semistable quiver representations are of Fano type.
\end{abstract}

\maketitle

\setcounter{tocdepth}{1} 
\tableofcontents

\section{Introduction}

The study of singularities is a central topic in algebraic geometry; especially, those which are quotients of smooth points by the action of an algebraic group play a fundamental role. 
By Nagata's negative answer to Hilbert's 14th problem~\cite{Nag61}, we know that rings of invariants
in finitely generated rings acted upon by algebraic groups are not always finitely generated.
Nevertheless, this statement holds if we restrict to the class of reductive groups in characteristic zero, which we will do in the following. Among \emph{reductive quotient singularities}, i.e., 
quotients of smooth points by the action of a reductive group, we find toric singularities, $\mathbb{T}$-singularities, and (finite) quotient singularities.
The understanding of toric singularities, which can be achieved using combinatorics, was crucial
for the development of toric geometry~\cite{CLS11}.
On the other hand, \textcolor{black}{finite} quotient singularities appear in orbifolds, or more generally, in the coarse moduli spaces of \textcolor{black}{smooth} Deligne-Mumford stacks~\cite{Kre09}. Analogously, reductive quotient singularities emerge in good moduli spaces of smooth Artin stacks~\cite{AHR20, Hol11}. 
In particular, reductive quotient singularities manifest when studying smooth projective varieties with reductive group actions, i.e., in the theory of $G$-varieties~\cite{Bri98, Bri03, Bri17}.

Another important class of singularities are \emph{Kawamata log terminal (klt)} singularities. These are the singularities of the minimal model program (MMP), \cite{Kol97, Kol13, KM98}, e.g.~in the sense that natural generalizations of the Kodaira Vanishing Theorem hold on klt spaces. From a complex-analytic viewpoint, the klt condition has a natural reformulation as an $L^2$-integrability condition. Although these kinds of singularities originate in birational geometry,  they have turned out to be of great importance also in moduli theory~\cite{Kol13b}, in the theory of algebraic $K$-stability~\cite{Xu20}, in geometric representation theory~\cite{KS14}, and in positive characteristic considerations, see e.g.~\cite{HW02, CST18}. 
The aim of this article is to prove that reductive quotient singularities are singularities of the minimal model program.

\subsection{Singularities of klt type and invariant-theoretic quotients}\label{subsect:mainresults1}
The singularities of the MMP are ---in a more comprehensive fashion than just for varieties--- defined for pairs $(X,B)$, where $X$ is a normal variety and $B$ is an effective divisor such that $K_X+B$ is $\qq$-Cartier. 
This effective divisor for example appears naturally when performing adjunction to a hypersurface in a klt variety, i.e., when looking for inductive arguments, or on base spaces of certain fibrations.
We will say that $X$ has \emph{klt type singularities} or is \emph{of klt type} if the pair $(X,B)$ is klt for some boundary $B$ on $X$. 
Quotient singularities and Gorenstein rational singularities are klt, see~\cite{Kol13} and~\cite{Kov00}, respectively, and toric singularities are of klt type~\cite{CLS11}. Moreover, finite quotients of klt singularities are again klt~\cite{Sho93}.

Thus, a natural question is whether the condition of being klt is preserved under quotients by reductive group actions. In~\cite{Scho05}, Schoutens proves that this is the case under the additional assumption that the quotient is $\qq$-Gorenstein.
However, as Example~\ref{ex:torus} shows, even a very simple $\mathbb{G}_m$ quotient of $\mathbb{A}^4$ can be non-$\mathbb{Q}$-Gorenstein.
Nonetheless, as the toric case shows, this can be remedied by adding a boundary through the singularity, which measures codimension one ramification of the quotient map. The main theorem of our article states that the class of klt type singularities is indeed preserved under reductive quotients.

\begin{introthm}\label{introtm:reductive-quot}
Let $X$ be an affine variety of klt type. Let $G$ be a reductive group acting on $X$.
Then, the (affine invariant-theoretic) quotient $X/\!/G:= \Spec \kk[X]^G$ is of klt type.
\end{introthm}

We emphasize that the boundary divisor making $X$ klt is not assumed to be $G$-invariant in the above statement. In particular, we obtain that any reductive quotient singularity, i.e., a singularity that is \'etale equivalent to the image of $0$ in the quotient $V/\!/G$ of a $G$-representation $V$, is of klt type; this applies to many classical examples of singular spaces such as generic determinantal varieties, see e.g.~\cite[Chap.~II]{ACGH85}, and especially Section~\ref{subsect:Collapsing} below.

\begin{introcor}\label{introcor:reduct-quot}
A reductive quotient singularity is of klt type.
\end{introcor}

Let us recall that reductive quotient singularities are known to be Cohen-Macaulay due to the Hochster-Roberts theorem~\cite{Hoc95}.
Even further, due to Boutot's theorem~\cite{Bou87}, we know that reductive quotient singularities are rational.
More generally, the reductive quotient of a rational singularity is again rational; see also \cite{Kov00a}.
Recently in~\cite{Mur20}, Murayama generalized this statement for locally quasi-excellent $\qq$-algebras. Any klt type singularity is rational~\cite{KM98}. 

However, klt type singularities are way better behaved than general rational singularities. As already mentioned above, many global vanishing theorems known in the smooth case continue to hold on projective varieties of klt type, \textcolor{black}{but fail to hold on projective varieties with more general rational singularities.}  Moreover, klt pairs satisfy certain local vanishing theorems, see \cite{Kol11b}. Regarding topological properties, we know that the local fundamental group of a klt type singularity is finite~\cite{Xu14,Bra20} (this allows to extend Nonabelian Hodge Theory to klt spaces \cite{GKPT19}) and that it satisfies the Jordan property~\cite{BFMS20}, while the local fundamental group of a rational singularity can be an arbitrary $\qq$-superperfect group~\cite{KK14}. 
In a similar vein, klt type singularities are local versions of Mori dream spaces~\cite{BM21}, i.e., their local Cox ring is finitely generated, while this is no longer the case for general rational singularities.
Finally, there are many interesting invariants of singularities that can be defined for klt type singularities, as the normalized volume~\cite{LLX20} and the minimal log discrepancy~\cite{Mor18b,Mor21b}.
In summary, the klt type property of $X/\!/G$
gives us a substantial advancement towards the understanding of reductive quotient singularities.

In the remainder of this introduction we will discuss applications of our main results stated above. 

\subsection{Good quotients of varieties of klt type} \label{subsect:mainresults2}
Since the invariant-theoretic quotients considered above are the local versions of good quotients, we obtain the following consequence of Theorem~\ref{introcor:reduct-quot}. 

\begin{introthm}\label{introthm:proj-quot}
Let $X$ be a variety of klt type. Assume that the reductive group $G$ acts on $X$ and there exists a good quotient $X\rightarrow X/\!/G$ with quasi-projective quotient space $X/\!/G$. Then,  $X/\!/G$ is of klt type.
\end{introthm}

This in turn applies in the setting of Geometric Invariant Theory (GIT) to yield the following statement. 

\begin{introcor}\label{introcor:GITquotients}
Let $X$ be a variety of klt type and let $L$ be a line bundle over $X$. Assume the reductive group $G$ acts on $X$ and that the action has been linearized in $L$. Then, the GIT-quotient $X^{ss}_L/\!/G$ is of klt type.
\end{introcor}

\textcolor{black}{We note that Theorem~\ref{introtm:reductive-quot} also implies that Hausen's GIT-quotients with respect to \textit{Weil divisors} \cite{Hau04}, which are not quasi-projective in general, are at least Zariski-locally of klt-type.}

\subsection{Symplectic quotients of K\"ahler manifolds}\label{subsect:symplectic}
Since the holomorphic action of a reductive group on a possibly non-algebraic K\"ahler manifold is locally algebraic if the action of a maximal compact subgroup is Hamiltonian due to a Luna-type slice theorem \cite[(2.7) Theorem]{HL94}, many results in the algebraic category have analytic counterparts; e.g., analytic versions of Boutot's result are proven in \cite{Gre10} and \cite{G11}. Also, our main result can be applied in this setup and leads to the following consequence.

\begin{introthm}\label{introthm:Hamiltonian}
Let $(X, \omega)$ be a K\"ahler manifold on which the reductive group $G$ acts holomorphically. Assume that the action of a maximal compact subgroup $K$ of $G$ admits a momentum map $\mu: X \to \mathrm{Lie}(K)^{\vee}$ with respect to the K\"ahler form $\omega$. Then, each point of the complex space $\mu^{-1}(0)/K$ obtained from $X$ by Symplectic Reduction admits an open neighborhood that is biholomorphic to a Euclidean open subset in an affine variety of klt type.
\end{introthm}

\subsection{Singularities of good moduli spaces}\label{subsect:goodmoduli}
Many moduli stacks have either local or even global descriptions as quotient stacks for actions of reductive groups on smooth varieties. If such a moduli stack admits a quasi-projective good moduli space in the sense of \cite{Alp13}, our main results can be applied to prove that the good moduli space has klt type singularities. As a consequence, similar to the Hassett-Keel program \cite{HassettHyeon} one could try to run a minimal model program on the good moduli space and interpret the steps of the MMP in terms intrinsic to the moduli problem, cf.~\cite{deVleming}.

One particular case admitting a global quotient description is the moduli stack of \textcolor{black}{K-polystable smooth Fano varieties}  of fixed volume~\cite{LWX18}, for which we obtain the following result; see however Remark~\ref{rem:badboundary}.

\begin{introthm}\label{thm:K-moduli}
Let $n$ be a positive integer
and $v$ be a positive rational number.
Let $X_{n,v}$ be the good moduli space 
of \textcolor{black}{K-polystable smooth Fano varieties} of dimension $n$ and volume $v$. Then, there exists an effective divisor
$B_{n,v}$ on $X_{n,v}$ so that the pair
$(X_{n,v},B_{n,v})$ has Kawamata log terminal singularities.
\end{introthm} 

The general result for smooth, but not necessarily global quotient stacks reads as follows.

\begin{introthm}\label{introthm:stacksmoduli}
Let $\mathcal{X}$ be a quasi-separated smooth algebraic stack of finite type over $\mathbb{K}$ with affine diagonal. If $\mathcal{X}$ admits a quasi-projective good moduli space $\mathcal{X} \to X$, then $X$ is of klt type.
\end{introthm}

\subsection{Good quotients of Mori Dream Spaces}\label{subsect:MDSs}

We call a normal variety $X$ with finitely generated divisor class group $\Cl(X)$ and finitely generated Cox ring
\[
R_X:= 
\bigoplus_{[D] \in \Cl(X)} 
\Gamma(X,\mathcal{O}_X(D)),
\]
a \emph{Mori Dream Space (MDS)}. As the name suggests, a projective variety with these properties behaves optimally with respect to the MMP;  e.g., see~the seminal article \cite{HK00} of Hu and Keel as well as the comprehensive work~\cite{ADHL15}. By now, the Cox ring and related notions have been generalized in multiple directions such as to local rings~\cite{BM21} as well as to algebraic stacks and to even more general objects~\cite{HMT20}. In the present work, we make (heavy) use of such generalized Cox rings, c.f.~Sections~\ref{subs:coxrings} and~\ref{sec:iteration}.

If $X$ is a projective variety \emph{of Fano type}, i.e., if there exists a divisor $B$ on $X$ such that the pair $(X, B)$ is klt and $-(K_X + B)$ is ample, it was shown in~\cite{BCHM10} that $X$ is a Mori Dream Space. Conversely, since a Mori Dream Space $X$ can be realized as a GIT-quotient of $\overline{X}:=\Spec R_X$ by a diagonalizable group, we see that if $\overline{X}$ is of klt type, then so is $X$. In fact, generalizing a basic result for cones \cite[Lem.~3.1(3)]{Kol13}, it has been shown in~\cite{Bro13, GOST15, KO15} that imposing the klt condition on the Cox ring is much stronger than doing so on $X$: a projective Mori Dream Space $X$ has Cox ring of klt type if and only if it is of Fano type. 
In particular, the proof in~\cite{GOST15} relates to earlier works of Smith~\cite{Sm00} and Schwede-Smith~\cite{SS10}, which use characteristic-$p$-methods to investigate varieties of Fano type and singularities of section rings. In~\cite[Sec.~7]{Sm00}, Smith investigates reductive quotients of varieties of Fano type and $F$-regular type. In particular, assuming ~\cite[Conj.~7.5]{Sm00}, Smith proves in~\cite[Thm.~7.6]{Sm00}, that GIT-quotients of varieties of Fano type are of globally $F$-regular type (and thus, by~\cite[Cor.~5.4]{GOST15} combined with~\cite[Thm.~1.1]{Bae11}, of Fano type). To be precise, Smith's conjecture asks whether affine algebras of characteristic 0 all of whose ideals are tightly closed are in fact of $F$-regular type; this relates also to more general questions that can be found in Section 8 of Hochster's Appendix 1 to~\cite{Hu96} and in~\cite[(2.5.7) Remark, (4.3.9) Remark]{HoHu99}.

The methods developed in the present paper yield the following statement, by which we may circumvent~\cite[Conj.~7.5]{Sm00} and deduce the ``Fano type" statement unconditionally as a corollary. 

\begin{introthm}
\label{introthm:MDSquotients}
Let $X$ be a Mori Dream Space. 
Let $G$ be a reductive group acting algebraically on $X$, and let $U \subseteq X$ be an open $G$-invariant subset admitting a good quotient $\pi\colon U \to U /\!/ G$.
Then, the quotient $U /\!/ G$ is a Mori Dream Space. Moreover, if $\Spec R_X$ is of klt type, then $\Spec R_{U/\!/G}$ is of klt type.
\end{introthm}

In case $X$ and $U/\!/G$ possess only constant invertible global functions, the Mori Dream Space property of $U/\!/G$ was already known due to~\cite{Bae11}. However, as was shown in~\cite{HMT20}, this assumption is only needed to guarantee \emph{uniqueness} of Cox rings, c.f.~Section~\ref{subs:coxrings}. Using the characterization~\cite[Cor.~5.3]{GOST15}, we then get the desired unconditional statement on reductive quotients of varieties of Fano type.

\begin{introcor}
\label{introcor:Fanoquotients}
Let $X$ be a projective variety of Fano type, i.e., assume there exists a boundary $B$ on $X$ such that $(X,B)$ is klt and $-(K_X+B)$ is ample. Let $L$ be an ample line bundle over $X$. Let a reductive group $G$ act on $X$, and assume that the action has been linearized in the line bundle $L$. Then, the GIT-quotient $X^{ss}_L/\!/G$ is of Fano type.
\end{introcor}
We remark that~\cite[Cor.~5.3]{GOST15} was reproved using the MMP in characteristic $0$ in~\cite[Thm.~4.1]{KO15}. Thus, the above statement does \emph{not} rely on characteristic-$p$-methods at all. 

The vanishing results for higher cohomology groups of nef line bundles on GIT-quotients of \textcolor{black}{smooth Fano varieties} discussed in~\cite{Sm00} and \cite{Scho05} then follow from the Kawamata-Viehweg vanishing theorem. Corollary~\ref{introcor:Fanoquotients} applies to many classically studied varieties such as moduli spaces of $n$ ordered (weighted) points on a line up to projective equivalence, \cite{HMSV12}. We emphasize that our result applies to GIT-quotients only; for example, normalized Chow quotients of Fano varieties do not need to be of Fano type, see Remark~\ref{rem:ChowQuotients} below. 

An interesting class of varieties obtained via GIT from invariant open subsets of $G$-representations (and hence of projective spaces) are moduli spaces of quiver representations, e.g., see \cite{Kin94} or \cite[Sect.~10.1]{Kir16}. Extending work done by Franzen, Reineke, and Sabatini in \cite{FRS20},  Corollary~\ref{introcor:Fanoquotients} immediately gives the following.

\begin{introcor}\label{introcor:quivers}
Let $\mathcal{Q}$ be a connected finite quiver having $k$ vertices and no oriented cycles. Fix a dimension vector $\mathbf{n} = (n_1, \dots, n_{k})$ as well as a semistability condition $\theta$, i.e., a character of $\mathrm{GL}_{n_1} \times \cdots \times \mathrm{GL}_{n_k}$ descending to the projectivization. Then, the  moduli space $M^{ss}_{\mathcal{Q}}(\mathbf{n}, \theta)$ of $\theta$-semistable representations of $\mathcal{Q}$ over $\mathbb{K}$ with dimension vector $\mathbf{n}$ is a projective variety of Fano type.
In particular, $M^{ss}_{\mathcal{Q}}(\mathbf{n},\theta)$ is a Mori Dream Space.
\end{introcor}

Similarly, the Fano type statement holds for all projective varieties ``in class $Q_G$" for some reductive group $G$, as introduced by Nemirovski in \cite{Nem13} and studied further in \cite[Sect.~4]{Gre15}. Among these are classical examples such as projectivizations of generic determinantal varieties. 

\subsection{Collapsing of homogeneous bundles}\label{subsect:Collapsing}
Many singular varieties studied in classical algebraic geometry or in relation with representation theory can be realized as  collapsings of  homogeneous vector bundles over  flag manifolds.\footnote{The authors sincerely thank Michel Brion for pointing them towards this application of the main results of this paper.}  This construction was introduced by Kempf in~\cite{Ke76} and has been investigated extensively mainly in characteristic $0$, e.g.~see \cite{Wey03, Man20}. Recently, Kempf's statements were generalized to arbitrary characteristic by L\H{o}rincz in~\cite{Lo20}.

The setting is the following. For a connected reductive group $G$ and a parabolic subgroup $P$, we consider a $G$-representation $W$ and a $P$-stable submodule $V \subset W$. The image of the homogeneous vector bundle $G \times_{P} V$ under the ``collapsing map", $[g, v] \mapsto g\cdot v \in W$, is closed, equal to $G \cdot V$, and therefore called the \emph{saturation} of $V$. More generally, instead of the submodule $V$ itself one may consider $P$-stable subvarieties $X \subset V$ and their saturation $G \cdot X \subset W$.

In characteristic $0$, Kempf showed that if the action of the unipotent radical $U_P$ of $P$ on $V$ is trivial (i.e., if $P$ acts \emph{completely reducibly} on $V$, \cite[Lem.~1]{Ke76}), then the image $G\cdot V$ of a birational collapsing map has rational singularities, see \cite[Thm.~3]{Ke76}. L\H{o}rincz on the other hand shows that $G \cdot V$ is even of strongly $F$-regular type. 
We note that in order to deduce from this statement the property of being of klt type, one would 
require an appropriate boundary $\Delta$ with $K_{G \cdot V} + \Delta$ being $\mathbb{Q}$-Cartier. Instead, given our main results, we are able to modify the proof of~\cite[Thm.~3.10]{Lo20} in such a way that we obtain the klt type property directly:

\begin{introthm}
\label{introthm:Collapsing}
Let $G$ be a connected reductive group
with a parabolic subgroup $P < G$.
Let $W$ be a $G$-representation
and $V\subset W$ a $P$-stable submodule such that $P$ acts completely reducibly on $V$.
Let $X\subset V$ be a $P$-stable subvariety.
If $X$ is of klt type, then the $G$-saturation $G\cdot X$ is of klt type. 
\end{introthm}
The reader is referred to \cite[Sect.~4]{Lo20} for a number of situations in which this result may be applied, e.g.~to study rank varieties for quiver representations or Griffiths-type vanishing theorems on Schubert varieties.

\subsection{Techniques used and outline of the paper}
We finish this introduction by explaining some of the techniques involved in the proof of Theorem~\ref{introtm:reductive-quot}.
First of all, to treat the quotient by $G$ we will quotient first by the derived subgroup of its identity component.
Thus, we will reduce to the cases of semisimple groups, torus actions, and finite groups.
In Section~\ref{sec:torus+finite}, 
we will prove the main theorem for the case of torus actions having a klt fixed point, and in full generality for finite groups.
This part of the article will rely on the theory of proper polyhedral divisors~\cite{AH06} and Ambro's canonical bundle formula~\cite{Amb04, FG12}.
In Section~\ref{sec:etale}, we will prove that being a klt type variety is an \'etale property. 
Using this statement and Luna's \'etale slice theorem~\cite{Lun73}, we may reduce to the situation as appearing locally around a fixed point.
In Lemma~\ref{le:factorial-semisimple-quotient}, we will give a criterion
for the factoriality of
the quotient of a factorial variety by a semisimple group. This lemma will allow us to deal with the semisimple part of the quotient.
However, a klt type singularity may not be factorial.
In order to fix this, we lift the action of $G$ to the iteration of Cox rings of the singularity~\cite{BM21}, as explained in Section~\ref{sec:iteration}. The final step in the iteration will be a factorial variety. This allows us to finish the proof of the results stated in Sections~\ref{subsect:mainresults1} and \ref{subsect:mainresults2} above in Section~\ref{subsect:proof_of_main_results}. The proofs of the applications stated in Sections~\ref{subsect:symplectic}, \ref{subsect:goodmoduli}, \ref{subsect:MDSs}, and \ref{subsect:Collapsing} above are given in Section~\ref{sect:applicationsproof}. 

\textcolor{black}{\subsection{Further developments and open questions}
After the preprint \cite{PreprintVersion} of this paper appeared on the arXiv in early November 2021, Zhuang gave an alternative proof of (a slight generalisation of) Theorem~\ref{introtm:reductive-quot} in a preprint posted on the arXiv at the end of August 2022, see \cite[Thm~1.1~/~Cor~1.2]{Zhuang}. Zhuang's line of argumentation also uses the idea of reducing to the case of rational Gorenstein singularities by local finite generation properties in order to apply Boutot's result on the rationality of reductive quotient singularities. However, in contrast to our focus on geometric aspects of invariant theory and the ultimate reliance on topological finiteness properties of klt type singularities, Zhuang mostly uses ring theory and a local-finite-generation characterisation of klt type singularities. Another incarnation of local finite generation (in terms of ultrafilters) was used by Takagi and Yamaguchi to deal with plt pairs in a preprint that appeared in December 2023, see \cite{TakagiYamaguchi}.}

At this point, it is natural to extend the focus to log canonical singularities when considering the problem of reductive quotients, i.e., we ask the following:

\begin{introquest} \label{Q1}
Let $(X,x)$ be a log canonical type singularity.
Let $G$ be a reductive group acting on $(X,x)$.
Is the quotient $(X/\!/G; [x])$
of log canonical type?
\end{introquest} 

Similar to the klt case, 
we expect that lc type singularities are preserved under reductive quotients, since a lot of the structural results follow a pattern that is similar to the klt type case. Indeed, log canonical singularities are Du Bois~\cite{KK14} and Du Bois singularities are preserved under reductive quotients. Also note that, analogous to the rational Gorenstein case, normal Du Bois singularities with $K_X$ Cartier are log canonical (see e.g.~\cite{Kov00}). On the other hand, log canonical singularities are not \textcolor{black}{local} Mori dream spaces in the sense of~\cite{BM21}, i.e., local finite generation may fail, making e.g.~lifting to the iteration of Cox rings impossible in the lc type case. \textcolor{black}{Together with the non-availability of the canonical bundle formula in the lc case, also in the other approaches mentioned above this seems to be a major complication.} Hence, new ideas are needed to approach Question~\ref{Q1}.

\subsection*{Acknowledgements} 
The authors would like to thank Michel Brion, Ana-Maria Castravet, St\'ephane Druel, Stefano Filipazzi, Jack Hall, Jochen Heinloth, Jesus Martinez-Garcia, Stefan Nemirovski, David Rydh, Karl Schwede, Hendrik S\"u\ss, Xiaowei Wang, and Chenyang Xu for many useful comments. Moreover, they are grateful to the referee for various suggestions on how to improve the exposition of the paper. LB and DG want to thank the organizers of the ``Komplexe Analysis -- Algebraicity and Transcendence" workshop at Oberwolfach Research Institute for Mathematics (MFO) for inviting us and the MFO for its excellent hospitality; our discussions on the topic of this paper started there in August 2020. 

\section{Preliminaries} 

Throughout this article, we work over an algebraically closed field $\mathbb{K}$ of characteristic zero.
By a point we always mean a closed point.
All \emph{varieties} are assumed to be irreducible.
All the divisors considered are $\qq$-divisors.
A \emph{reductive group} is a linear algebraic group $G$
whose connected component $G^0$ containing the identity  
has no non-trivial connected unipotent normal subgroups.

In the present section, we 
make some observations about equivariant geometry and
recall some preliminary results 
about klt singularities, Cox rings, and varieties with torus actions.

\subsection{Equivariant geometry} 
After recalling the notion of a good quotient that is central to this work, we prove a basic result about the existence of certain invariant neighborhoods of fixed points.

\begin{definition}
{\em 
Let $X$ be a separated scheme of finite type acted upon by a reductive group $G$. A surjective $G$-invariant morphism $\pi: X \to Y$ to a separated scheme $Y$ is called a \emph{good quotient} if
\begin{enumerate}
    \item $\pi$ is an affine morphism, and
    \item
    the natural pull-back homomorphism
    $\pi^*\colon \mathcal{O}_Y \rightarrow (\pi_*\mathcal{O}_X)^G$ is an isomorphism.
\end{enumerate}
Here, sections of the sheaf $(\pi_*\mathcal{O}_X)^G$ over an open subset $U$ in $Y$ are the $G$-invariant regular functions on $\pi^{-1}(U)$. If a good quotient exists, it is unique and will be denoted by $X /\!/G$.}
\end{definition}

The most prominent examples of good quotients are obtained as follows: if $V$ is a finite dimensional representation of the reductive group $G$ over $\mathbb{K}$, then the ring $\mathbb{K}[V]^G$ of invariants under this group action is finitely generated by a famous result of Hilbert, and the induced map from $V$ to the spectrum of $\mathbb{K}[V]^G$ is a good quotient; e.g., see \cite{AGIV}. Similarly, if $X$ is an affine $G$-variety, then a good quotient $X/\!/G = \mathrm{Spec}(\mathbb{K}[X]^G)$ always exists as an affine variety. Quotients obtained from Mumford's Geometric Invariant Theory (GIT) are likewise good quotients, see \cite{MFK94}.
\begin{lemma}
\label{lem:shrinking-around-fixed-point}
Let $X$ be a normal affine variety and $G$ a reductive group acting on $X$. 
Let $x \in X$ be a $G$-fixed point. Then the following hold:
\begin{enumerate}
    \item If $\mathcal{O}_{X,x}$ is factorial, then there is a $G$-stable \emph{locally factorial} principal open affine neighborhood $x \in X_g \subseteq X$, i.e., $\mathcal{O}_{X,y}$ is factorial for any $y \in X_g$.
    \item When $D \in \CaDiv(X)$ is $G$-invariant, then there is a $G$-stable principal open affine neighborhood $x \in X_g \subseteq X$, such that $\left. D \right|_{X_{g}}$ is principal.
\end{enumerate}
Moreover, in any of the two cases, the image $\pi(X_g)$ under the invariant theoretic quotient morphism $\pi\colon X \to X/\!/G$ is an affine open principal neighborhood of $\pi(x)$. 
\end{lemma}

\begin{proof}
We prove the first item. Let $V \subseteq X$ be the locus where $X$ is locally factorial, i.e., $y \in X$ is in $V$ if and only if $\mathcal{O}_{X,y}$ is factorial. Since $\mathcal{O}_{X,x}$ is factorial by assumption, $V$ contains $x$, and by~\cite[Item~(1), p.~21]{BF84}, $V$ is open. Moreover, $V$ is $G$-stable, since $\mathcal{O}_{X,g.y}$ is factorial if and only if $\mathcal{O}_{X,y}$ is factorial. Thus $Y=X\setminus V$ is $G$-stable and closed in $X$. So we can find an invariant element $0 \neq f \in I(Y)$ in the ideal of $Y$ with $f(x) \neq 0$. In particular, $f$ is a regular function on $X/\!/G$ and thus $\pi(X_f)=(X/\!/G)_f=X_f/\!/G$. 

We come to the second item. Since $D$ is Cartier and $G$-invariant, we obtain a corresponding \emph{canonically $G$-linearized line bundle} $L=\mathcal{O}_X(D)$ on $X$, cf.~\cite[Prop.~1.7]{Hau04}. In particular, the action on sections $f \in \mathcal{O}_X(D) \subseteq \kk(X)$ is given by $g.f(y)=f(g.y)$ for all $y \in X$. By this definition, the fact that $x$ is fixed by the $G$-action on $X$ means that $G$ acts on the fiber $L(x)$ trivially. Since $X$ is affine, the pullback map $H^0(X,L) \to L(x)$ is surjective. Since the representation of $G$ on $H^0(X,L)$ is locally finite, \textcolor{black}{there exists a finite-dimensional $G$-stable subspace $W \subseteq H^0(X,L)$ mapping to $L(x)$. Since $G$ is reductive, $W$ is completely reducible, i.e., it is a direct sum of irreducible $G$-representations. Take one of these, say $W' \subseteq W$, which maps to $L(x)$ as well. By Schur's Lemma, $L':=W' \cong  L(x)$ as $G$-representations.} So, we found a (pointwise) $G$-invariant line $L' \subseteq H^0(X,L)$ such that the composition
  $$
  L' \hookrightarrow H^0(X,L) \to L(x)
  $$
  is surjective. Now let $0 \neq s \in L'$, then $s(x) \neq 0$. Let 
  \[
  U':=\{ y \in X\mid  s(y) \neq 0\}.
  \]
  Then $U'$ is open, affine, and $G$-stable in $X$, contains $x$ and $\left.L\right|_{U'}$ is trivial. Thus $\left.D\right|_{U'}$ is principal. Arguing as before, we can find an invariant $0 \neq g \in I(X \setminus U')$ with $g(x) \neq 0$. Then $x \in X_g \subseteq X$ and $D$ is principal on $X_g$. 
  As before, $\pi(X_g)=(X/\!/G)_g=X_g/\!/G$ is an affine open principal neighborhood of $[x]$.
\end{proof}

\subsection{Kawamata log terminal singularities}

In this subsection, we recall the notion and basic properties of Kawamata log terminal singularities.

\begin{definition}
{\em 
A {\em sub-pair}
is a couple $(X,B)$ 
where $X$ is a normal quasi-projective variety
and $B$ is a $\qq$-divisor
so that $K_X+B$ is a $\qq$-Cartier $\qq$-divisor.
A {\em pair} (or {\em log pair}) is a sub-pair
with $B$ effective.
}
\end{definition}

\textcolor{black}{
The above convention is common in the recent literature on the subject and it helps in keeping track of (non-)effectivity in longer arguments. Note however that the standard reference \cite{KM98} does not distinguish between pairs and sub-pairs. }

\begin{definition}
{\em 
Let $X$ be a normal quasi-projective variety.
A {\em prime divisor over $X$}
is a prime divisor which lies in a normal variety
admitting a projective birational morphism to $X$.
This means that we can find a projective birational morphism
$\pi\colon Y\rightarrow X$ so that
$E\subset Y$ is a prime divisor.
The {\em center} of $E$ on $X$,
denoted by $c_X(E)$, is the image of $E$ on $X$.

Let $(X,B)$ be a sub-pair, 
\textcolor{black}{$\pi\colon Y\rightarrow X$ a projective birational morphism, 
and pick $K_Y$ so that $\pi_*K_Y=K_X$.}
Let $E$ be a prime divisor over $X$.
The {\em log discrepancy} of $(X,B)$ at $E$
is defined to be
\[
a_E(X,B):=1+{\rm coeff}_E(K_Y-\pi^*(K_X+B)).
\]
}
\end{definition} 

\begin{definition}
{\em 
Let $(X,B)$ be a sub-pair.
A {\em log resolution} of $(X,B)$ is 
a projective birational morphism
$\pi\colon Y\rightarrow X$ satisfying the following conditions:
\begin{enumerate}
    \item $Y$ is a smooth variety,
    \item the exceptional locus
    of $\pi$ is purely divisorial, and
    \item the divisor
    $\pi^{-1}_*B+{\rm Ex}(\pi)_{\rm red}$ is a reduced
    divisor with simple normal crossing.
\end{enumerate}
}
\end{definition}
By Hironaka's resolution of singularities,
we know that any sub-pair admits a log resolution.

\begin{definition}
{\em 
Let $(X,B)$ be a sub-pair.
We say that $(X,B)$ is {\em sub Kawamata log terminal} (or {\em sub-klt} for short) if all its log discrepancies are positive; i.e., $a_E(X,B)>0$ for every prime divisor $E$ over $X$.
It is known that a log pair $(X,B)$ is sub-klt if and only if all 
the log discrepancies corresponding 
to prime divisors on some log resolution of $(X,B)$ are positive~\cite[Cor.~2.32]{KM98}.
If $(X,B)$ is sub-klt and additionally a log pair,
then we say that it is {\em Kawamata log terminal} 
(or {\em klt}).
If $X$ is a quasi-projective normal variety
and there exists a boundary $B$ for which $(X,B)$ is klt,
then we say that $X$ is of {\em klt type}. If $(X, \emptyset)$ is klt, we say that $X$ is \emph{klt}; in this case, often $X$ is said to have \emph{log terminal singularities}. 
If $a_E(X, \emptyset) \geq 1$ for every prime divisor $E$ over $X$, we say that $X$ has \emph{canonical} singularities.
}
\end{definition} 

\subsection{Preliminaries on Cox rings}
\label{subs:coxrings}

In this subsection, we recall the necessary preliminaries regarding Cox rings and related notions. Depending on the generality and the precise setting, we refer to~\cite{ADHL15, BM21, HMT20}. 
We recall from~\cite[Def 1.3.1.1]{ADHL15} the following basic definition.

\begin{definition}
{\em 
Let $X$ be a normal algebraic variety and $K \subseteq \WDiv(X)$ be a finitely generated subgroup. The \emph{sheaf of divisorial algebras} associated to $K$ is the quasi-coherent sheaf of $K$-graded $\mathcal{O}_X$-algebras 
$$
\mathcal{S}_K:= 
\bigoplus_{D \in K} 
\mathcal{O}_X(D).
$$
}
\end{definition}

When such a sheaf is \emph{locally of finite type}, that is, when every point $x \in X$ has an open affine neighborhood $x \in U \subseteq X$ with $\Gamma(U, \mathcal{S}_K)$ a finitely generated $\kk$-algebra, then the \emph{relative spectrum} $\widetilde{X}:=\Spec_X \, \mathcal{S}_K$  is a  normal variety by \cite[Prop 1.3.2.8]{ADHL15}.  Moreover, the morphism $\widetilde{X} \to X$ is a  good quotient for the action of the quasi-torus (or diagonalizable group) $H':= \Spec \, \kk[K]$ on $\widetilde{X}$ that is induced by the $K$-grading~\cite[Constr.~1.3.2.4]{ADHL15}. This quotient does not extract a divisor (see, e.g.,~\cite[Prop.~1.6.1.6]{ADHL15}).

Finite generation properties of such sheaves are of great interest. In particular, one is interested in the finite generation of the ring of global sections $S:=\Gamma(X,\mathcal{S}_K)$, which in general is much stronger than $\mathcal{S}_K$ being locally of finite type. The following characterization relates global finite generation with local finite generation.

\begin{proposition}
\label{prop:finite-type-local}
Let $X$ be a normal algebraic variety and $K \subseteq \WDiv(X)$ be a finitely generated subgroup. Let $\mathcal{S}_K$ be the sheaf of divisorial algebras associated to $K$. Then, the following statements are equivalent:
\begin{enumerate}
    \item The sheaf $\mathcal{S}_K$ is locally of finite type.
    \item For all $x \in X$, the stalk $\mathcal{S}_{K,x}$ is a finitely generated $\mathcal{O}_{X,x}$-algebra.
\end{enumerate}
If in addition $X$ is affine, then (1) and (2) are also equivalent to
\begin{enumerate}
    \item[(3)] The ring of global sections $S:=\Gamma(X,\mathcal{S}_K)$ is a finitely generated $\kk$-algebra.
\end{enumerate}
\end{proposition}

\begin{proof}
The equivalence of (1) and (2) is~\cite[Lemma 2.19]{BM21}. The proof of~\cite[Prop.~4.3.1.1]{ADHL15} shows the implication (3) $\Rightarrow$ (1), \textcolor{black}{ which does not require $X$ to be affine, but holds in general.}

\textcolor{black}{Now let $X$ be affine. The implication (1) $\Rightarrow$ (3) follows since $\widetilde{X} \to X$ is an affine morphism, which means that $\widetilde{X}$ is affine and $S$ equals the ring of global sections $\Gamma(\widetilde{X},\mathcal{O}_{\widetilde{X}})$.}
\end{proof}

One usually speaks of a \emph{Cox sheaf} ---and \emph{Cox ring} for the ring of global sections--- when $K$ is mapped isomorphically to $\Cl(X)$. Finite generation of the Cox ring is of great interest in birational geometry, since a variety with this property allows only finitely many small birational modifications, controlled by the \emph{variation of GIT} on the Cox ring. In particular, any MMP on such a variety terminates, and consequently, it is called a \emph{Mori Dream Space}.  

If we want to construct Cox sheaves or Cox rings, \textcolor{black}{in the case the class group is a free finitely generated abelian group} we only have to choose a subgroup $K \subseteq \WDiv(X)$ such that the projection $\WDiv(X) \to \Cl(X)$ restricts to an isomorphism $K \xrightarrow{\cong} \Cl(X)$. Then~\cite[Constr.~1.4.1.1]{ADHL15} shows that the Cox sheaf
\[
\mathcal{R}_X:= 
\bigoplus_{[D] \in \Cl(X)} 
\mathcal{O}_X(D),
\]
where $D \in K$ is a representative of the class $[D]$, does not depend on the choice of $K$, up to isomorphism. 

%The same construction may of course be performed with $\Cl(X)$ replaced with e.g., the local class group $\Cl(X_x)$ of some $x \in X$ (as long as this group is free). 

On the other hand, if the class group has torsion, one needs an additional quotient construction, which may introduce non-uniqueness of the Cox ring in certain cases. However, as~\cite{HMT20} shows and as we will discuss subsequently, the Cox rings are still well-defined. The following construction proceeds along the lines of~\cite[Constr.~1.4.2.1]{ADHL15}.

\begin{construction}\label{constr:CoxSheaf}
{\em 
Let $X$ be a normal algebraic variety and $K \subseteq \WDiv(X)$ be a finitely generated subgroup.
Let $L \subseteq K$ be a subgroup
contained in the kernel of $K\rightarrow {\rm Cl}(X)$. \textcolor{black}{Due to~\cite[Constr.~1.4.2.3]{ADHL15}, there exists a character $\chi \colon L \to \kk(X)^*$ satisfying}
$$
\ddivv(\chi(D))=D
\qquad
\text{ for all } D \in L.
$$
Denote by $\mathcal{I}_{L,\chi}$ the sheaf of ideals of $\mathcal{S}_K$ locally generated by the sections $1-\chi(D)$ for $D \in L$, \textcolor{black}{where $1$ is homogeneous of degree $0$ and $\chi(D)$ is homogeneous of degree $-D$}. 

The \emph{Cox sheaf associated to $K, L$, and $\chi$} is the sheaf $\mathcal{R}_{X,K,L,\chi}:=\mathcal{S}_K/\mathcal{I}_{L,\chi}$, which has a grading by $M:=K/L$ defined by
$$
(\mathcal{R}_{X,K,L,\chi})_{m}:=\pi \left( \bigoplus_{D \in c^{-1}(m)} (\mathcal{S}_{K})_{D} \right),
$$
where $\pi \colon \mathcal{S}_{K} \to \mathcal{R}_{X,K,L,\chi}$ and $c \colon K \to M$ are the projections. \textcolor{black}{The \emph{Cox ring associated to $K, L$, and $\chi$} is its ring of global sections
$
R_{X,K,L,\chi}:=\Gamma(X,\mathcal{R}_{X,K,L,\chi}).
$}
}
\end{construction}

\begin{remark}
\label{rem:nonconstant-global-functions}
{\em 
The original construction \cite[Constr.~1.4.2.1]{ADHL15} is a special case of the above, where the restriction of the projection $\WDiv(X) \to \Cl(X)$  to $K$ is assumed to be surjective and $L$ is defined to be the kernel of this projection; i.e., $M$ is isomorphic to $\Cl(X)$. Moreover, $X$ is assumed to have only constant global invertible functions. In this case, $\mathcal{R}_{X,K,L,\chi}$ and $R_{X,K,L,\chi}$ are the usual Cox sheaf and Cox ring of $X$, respectively. 

In~\cite[Theorem B]{HMT20}, it was shown that when $X$ is allowed to have nonconstant global invertible functions, the set of isomorphism classes of Cox rings \textcolor{black}{ $R_{X,K,L,\chi}$ as above with $M \cong \Cl(X)$} is in bijection to $\mathrm{Ext}^1(\Cl(X), \Gamma(X, \mathcal{O}^*))$. So, for example in the case of local rings of singularities as considered in~\cite{BM21}, Cox rings are not unique in general.
}
\end{remark}

In the following, we will be interested in a specific choice for the subgroups of Weil divisors $K$ and $L$ adjusted to the local setting (cf.~\cite[Def.~3.15]{BM21}):

\textcolor{black}{
\begin{definition}
For a point $x$ of a normal affine variety $X$, we denote by $\PDiv(X,x)$ the group of all Weil divisors principal on some neighborhood of $x$, cf.~\cite[Sec. 1.6.2]{ADHL15}. We call $\Cl(X,x):=\WDiv(X)/\PDiv(X,x)$ the local class group of $x \in X$.
\end{definition} 
}

Note that when $X_x:=\Spec \mathcal{O}_{X,x}$ is the spectrum of the Zariski local ring of $x \in X$, then $\Cl(X,x)$ is equal to the class group of $X_x$.  The following result is crucial in order for us to be able to apply Construction~\ref{constr:CoxSheaf}.

\begin{proposition}
\textcolor{black}{Let $x \in X$ be a klt type singularity. Then the abelian group $\Cl(X,x)$ is finitely generated.} 
\end{proposition}
\begin{proof}
We denote by $\mathcal{O}_{X,x}^{h}$ the Henselization and by $\widehat{\mathcal{O}}_{X,x}$ the completion of the local ring $\mathcal{O}_{X,x}$. 
The ring $\mathcal{O}_{X,x}$ is rational due \textcolor{black}{to~\cite[Thm.~1-3-6]{KMM87}}. \textcolor{black}{In particular, it is $k$-rational for every $k>0$, that is, for every resolution $f: Y \to X$, the direct image sheaf $R^k f_{*} \mathcal{O}_Y$ vanishes around $x$.}

Due to $1$-rationality, by~\cite[Thm 6.1]{BF84}, the class group $\Cl(\mathcal{O}_{X,x}^{h})$ is finitely generated. 
Since $\mathcal{O}_{X,x}$ is $2$-rational, so is $\mathcal{O}_{X,x}^{h}$, and therefore the natural homomorphism \[\Cl(\mathcal{O}_{X,x}^{h}) \to \Cl(\widehat{\mathcal{O}}_{X,x})\] is bijective by~\cite[Thm 6.2]{BF84}. 

On the other hand, by~\cite[Thm.~(6.5)]{Sam64} (Mori's Theorem), the homomorphism from $\Cl(\mathcal{O}_{X,x})$ to $\Cl(\widehat{\mathcal{O}}_{X,x})$ realizes $\Cl(\mathcal{O}_{X,x})$ as a subgroup of a finitely generated abelian group, which concludes the proof.
\end{proof}

\textcolor{black}{In the following, we will discuss the crucial local finite-generation property for Cox rings of klt type singularities.}

\begin{definition} 
{\em 
\label{def:aff-loc-Cox-ring}
Let $X$ be an affine klt type variety and $x \in X$. 
Let $K \subseteq \WDiv(X)$ be finitely generated by divisors $D_1,\ldots,D_k$ going through $x$, such that the restriction $\varphi \colon K \to \Cl(X,x)$ of the quotient $\WDiv(X) \to \Cl(X,x)$ is surjective. Let $L:=\ker(\varphi)$ and fix a character $\chi \colon L \to \kk(X)^*$ with $\ddivv(\chi(D))=D$ for all $D  \in L$. We call
$$
R_{X,x}:=\Gamma(X,\mathcal{R}_{X,K,L,\chi})
$$
the \emph{affine local Cox ring of $X$ at $x$} (with respect to $K,L,$ and $\chi$, but we will omit this dependency, as it will play no role in the following) or \emph{aff-local Cox ring} for short. We call $X':=\Spec R_{X,x}$ the associated \emph{Cox space}.
}
\end{definition} 
\textcolor{black}{
In ~\cite{BM21}, two of the authors of this paper defined \emph{local Cox rings} of klt singularities. These Cox rings are of the form $R_{X_x}:=\Gamma(X_x,\mathcal{R}_{X_x,K,L,\chi})$, where $K/L \cong \Cl(X,x)$. These Cox rings are not local in general, but graded-local. That is,  they possess a unique graded maximal ideal~\cite[Sec 2.2]{BM21}. Graded-local rings can be obtained by a process similar to localization. The \emph{graded-localization} of a graded algebra  $A$ at a graded prime $\mathfrak{p}$ is defined to be $A_{(\mathfrak{p})}:=S^{-1}A$, where the multiplicative set $S$ consists of homogeneous elements of $A \setminus \mathfrak{p}$.
The ideal $\langle \mathfrak{p} \rangle$ generated by $\mathfrak{p}$ in $A_{(\mathfrak{p})}$ is the unique graded maximal one and localizing again with the multiplicative set $A_{(\mathfrak{p})} \setminus \langle \mathfrak{p} \rangle$ gives the usual localization, so we have inclusions
$$
A \hookrightarrow A_{(\mathfrak{p})} \hookrightarrow A_{\mathfrak{p}} = \left( A_{(\mathfrak{p})} \right)_{\langle \mathfrak{p} \rangle}.
$$
Going from $A_{(\mathfrak{p})}$ to $A_{\mathfrak{p}}$ preserves the divisor class group~\cite[Lem. 2.9]{BM21}.
As the name suggests, the aff-local Cox ring from Definition~\ref{def:aff-loc-Cox-ring} can be seen as an affine version of the local Cox ring, and the following proposition shows that indeed the graded-local version can be obtained from it by graded-localization.}

\begin{proposition}
\label{prop:aff-loc-Cox-ring}
Let $X$ be an affine klt type variety
and  $x\in X$.
Then, the following statements hold:
\begin{enumerate}
    \item $R_{X,x}$ is a finitely generated $\kk$-algebra. In particular,  $X'$ is an affine variety.
    \item $X'$ is of klt type. 
    \item The quasi-torus action on $X'$ is strongly stable in the sense of \cite[Def.~1.6.4.1]{ADHL15}. 
    \item The quotient $X' \to X$ has exactly one invariant point $x'$ in the preimage of $x \in X$. 
    \item The graded-localization of $X'$ at $x'$ equals the spectrum of the local Cox ring
    $$
    R_{X_x}:=\Gamma(X_x,\mathcal{R}_{X_x,K,L,\chi})
    $$
\end{enumerate}
\end{proposition}

\begin{comment}

\begin{remark}
{\em
\label{rem:klt-ness-Cox}
Apart from the finite generation claim, Proposition~\ref{prop:aff-loc-Cox-ring} makes a statement about klt-ness of (the spectrum of) the Cox ring. Similar to the question of finite generation, this is analogous to the projective case, where~\cite{Bro13, GOST15} shows that the Cox ring of a Fano variety has klt singularities. Indeed, the local and the global statement are equivalent.
While~\cite{GOST15} uses reduction to characteristic $p$, the proof in~\cite{Bro13} reduces the question iteratively to line bundles. In~\cite[Lem.~3.22, Thm.~3.23, Cor.~3.24]{BM21}, the statement is proven with a flavor similar to~\cite{Bro13}, but in the relative and, in particular, in the affine setting.
The proof follows the ideas of~\cite{LS13};
we use the klt-ness of the singularity to prove that the normalized Chow quotient of the torus action is of Fano type over the GIT quotient.
}
\end{remark}

\end{comment}

\begin{proof}
We begin with Item (1). Since $X$ is affine and of klt type, the Cox sheaf associated to $\Cl(X,x)$ is finitely generated over $\mathbb{K}[X]$ due to~\cite[Cor.~1.1.9]{BCHM10}, which says that the pushforward of a Cox sheaf along a morphism of Fano type is finitely generated. Thus, the claim follows from the equivalence of (1) and (3) of Proposition~\ref{prop:finite-type-local}.

\textcolor{black}{
The proof of Item (2) can be copied verbatim from the proof of~\cite[Thm.~3.26]{BM21}, with
$\phi\colon X\to Z$ the identity 
and $\Cl(X,\Delta)$ therein replaced by $\Cl(X,x)$.} 

The proof of Item (3) follows in the same way as it is indicated in~\cite[Rem.~1.6.4.2]{ADHL15}. Tracing back the line of arguments, the relevant items from~\cite[Rem.~1.6.1.6]{ADHL15} are (i) and (ii). These  follow from~\cite[Lem.~1.5.1.2, Constr.~1.5.1.4]{ADHL15}. These in turn rely on~\cite[Constr.~1.4.2.1]{ADHL15}, where the requirement of having only constant globally invertible functions is only needed to guarantee the uniqueness of the Cox ring. Thus we can replace~\cite[Constr.~1.4.2.1]{ADHL15} by our Construction~\ref{constr:CoxSheaf} and use the same reasoning to see that the quasi-torus $H:=\Spec \kk[\Cl(X,x)]$ acts strongly stably on $X'$  (cf. also Remark~\ref{rem:nonconstant-global-functions}).  

We come to Item (4). Since $R:=\kk[X]$ is the degree-zero part of $R_{X,x}$ with respect to the $\Cl(X,x)$-grading, the morphism $\varphi \colon X' \to X$ is a good quotient by $H$. In particular, the preimage $\varphi^{-1}(x)$ contains exactly one closed orbit. Let $x'$ be a point in this orbit. Then \cite[Prop. 1.6.2.5]{ADHL15} with $\Cl(X)$ therein replaced by $\Cl(X,x)$ shows that the isotropy group of $x'$ is the whole of $H$, i.e., $x'$ is $H$-invariant. Of course there can be no other $H$-invariant point in $\varphi^{-1}(x)$.

We turn to show Item (5). First, note that in Definition~\ref{def:aff-loc-Cox-ring}, we chose $K$ to be generated by Weil divisors $D_i$ on $X$ going through $x$. By abuse of notation, we can consider them (and the subgroup $L$) as a group of Weil divisors on $X_x$. The character $\chi \colon L \to \kk(X)^*=\kk(X_x)^*$ can also be transferred to the local setting, so the local Cox ring $R_{X_x}$ is well defined and can be obtained from $R_{X,x}$ by the base change $X_x \to X$, cf.~\cite[Pf. of Cor. 3.24]{BM21}. In other words, it is the stalk of the sheaf  $\mathcal{R}_{X,K,L,\chi}$ at $x$. 

Now, let $S$ be the set of non-units of $R \setminus \mathfrak{m}_{x}$ and let $S'$ be the set of homogeneous non-units of $R_{X,x}\, \setminus \mathfrak{m}_{x'}$. Here $\mathfrak{m}_{x}$ and $\mathfrak{m}_{x'}$ are the maximal ideals corresponding to $x$ and $x'$, respectively. Let $f \in S'$. Since $x'$ lies in the closure of every $H$-orbit mapping to $x$, the  (closed) image of the $H$-invariant principal divisor $Z:=V_{X'}(f)$ does not meet $x$. But since the spectrum of the local Cox ring $R_{X_x}$ is the fiber product $X' \times_{X} X_x$, we infer that $f$ must be a unit in $R_{X_x} \cong S^{-1}R \otimes_{R} R_{X,x}$. Since we have canonical monomorphisms
$$
S^{-1}R \otimes_{R} R_{X,x} \hookrightarrow (S')^{-1} R_{X,x} \hookrightarrow Q(R_{X,x}),
$$
this observation together with the universal property of localizations shows that the first map is in fact an isomorphism, as claimed. 
\end{proof}

\subsection{Preliminaries on torus actions}

This section introduces background on torus actions.
We will use standard toric notations. Namely, $N\simeq \zz^{n}$ is a lattice, $M =  \Hom_{\zz}(N, \zz)$
is its dual lattice, $M_{\qq}$ and $N_{\qq}$ are the associated $\qq$-vector spaces obtained 
by tensoring $M$ and $N$ by $\qq$ and $\TT:= \G_{m}\otimes_{\zz} N$ the associated algebraic torus.
We write $M_{\qq}\times N_{\qq}\rightarrow \qq$, $(m, v)\mapsto \langle m, v\rangle$ for the natural pairing.

First, we give an example for the phenomenon that quotients by torus actions may lead to non-$\mathbb{Q}$-Gorenstein singularities. 

\textcolor{black}{
\begin{example}
\label{ex:torus}
{\em 
Consider the action of $\G_{m}$ on four-dimensional affine space by
 $$
 t\cdot (x_1,\ldots,x_4):=(t^2\cdot x_1,t^{-1}\cdot x_2,t^{-1}\cdot x_3,t\cdot x_4).
 $$
 It is easy to see that the ring of invariants $\cc[x_1,\ldots,x_4]^{\G_{m}}$ with respect to this action is generated by the five invariants
 $
 x_1x_2^2$, $x_1x_2x_3$, $x_1x_3^2$, $x_2x_4$, and $x_3x_4$,
 which we denote by $f_1,\ldots,f_5$. 
 The ideal of relations is generated by 
 $
 f_1f_3-f_2^2$, $f_1f_5-f_2f_4$, and $f_2f_5-f_3f_4$.  The three-dimensional quotient $X:=\cc^4/\!/\G_{m}$ is the toric singularity given by these three relations or, equivalently, by the cone with primitive ray generators given by 
 $$
 u_1=(0,0,1), \quad u_2=(0,1,2), \quad u_3=(1,0,1), \quad u_4=(1,1,1)
 $$
 in $N=\zz^3$. In particular, the quotient is not a complete intersection. Since the primitive ray generators do not lie in a common affine hyperplane, no multiple of the canonical divisor $K_X=-(D_1+\ldots+D_4)$ can be written in the form \begin{equation}\label{localmarker}\sum \, \langle u_i,m\rangle D_i\end{equation} for some $m \in M$. Hence, the singularity is not $\qq$-Gorenstein by~\cite[Thm~4.2.8]{CLS11}. It is though of klt type. Indeed, choosing $D=\frac{1}{2}D_1$ as boundary, two times the log canonical divisor $K_X+D$ can be written in the above form \eqref{localmarker} with $m=(-1,0,-1)$; i.e., $K_X + D$ is $2$-Cartier. Inserting the ray through $(1,1,2)$ provides a log-resolution of $(X,D)$ such that the discrepancy of the sole exceptional divisor is $1/2$.
 }
\end{example}
}

\subsubsection{Polyhedral divisors and affine $\mathbb{T}$-varieties}
\textcolor{black}{In this subsection, we recall  the notion of polyhedral divisor and its connection to $\mathbb{T}$-varieties. Main references for the material summarized in this section are~\cite{AH06,AIPSV12}.
}
Given a polyhedron $\Delta\subseteq N_{\qq}$, we set 
\[
{\rm tail}(\Delta) = \{ v\in N_{\qq}\,|\, v + \Delta \subset \Delta\}
\]
and call it the \emph{tail-cone} of $\Delta$. For any strictly convex polyhedral cone $\sigma \subset N_{\qq}$ we consider the semigroup
$$\pol_{\sigma}(N_{\qq}) =  \{ \Delta \subset N_{\qq}\,|\, \Delta \text{ polyhedron with ${\rm tail}(\Delta)$}=  \sigma\}$$
with addition the Minkowski sum. Note that $\sigma$ is the neutral element of $\pol_{\sigma}(N_{\qq})$.
We also define the extended semigroup $\pol_{\sigma}^{+}(N_{\qq}) := \pol_{\sigma}(N_{\qq}) \cup \{\emptyset\}$,
where $\emptyset$ is an absorbing element, \textcolor{black}{i.e., $\Delta+\emptyset=\emptyset$ for any $\Delta\in \pol_{\sigma}^+(N_\qq)$}. Given a normal variety $Y$, we denote by $\CaDiv_{\geq 0}(Y)$ the semigroup
of effective Cartier divisors on $Y$. A \emph{$\sigma$-polyhedral divisor} over $(Y, N)$ is an element 
$$\D \in \pol_{\sigma}^{+}(N_{\qq})\otimes_{\zz_{\geq 0}}\CaDiv_{\geq 0}(Y).$$
Such a $\D$ admits a decomposition $\D  =  \sum_{Z\subset Y} \D_{Z}\otimes [Z]$, where the sum runs over the prime divisors $Z$ of $Y$,  $\D_{Z}\in \pol_{\sigma}^{+}(N_{\qq})$ for any $Z$,  \textcolor{black}{and for all but finitely many prime divisors $Z$ we have $\mathcal{D}_Z=\sigma$.}
The cone $\sigma$ is called the \emph{tail-cone} of $\D$ and is denoted by $\sigma(\D)$. The \emph{locus} of $\D$ is the open subset
$$\loc(\D):= Y\, \setminus\,  \bigcup_{Z\subset Y, \D_{Z} =  \emptyset}Z.$$
The map
$$\D:\sigma^{\vee}\rightarrow \CaDiv_{\qq}(\loc(\D)),\,\, m \mapsto \D(m):= \sum_{Z\subset Y, \D_{Z} \neq  \emptyset}\min\,\,\langle \D_{Z}, m\rangle\cdot  Z_{|\loc{\D}},$$
where $\CaDiv_{\qq}(\loc(\D))$ is the $\qq$-vector space of Cartier $\qq$-divisors on $\loc(\D)$ and $\sigma^{\vee}\subset M_{\mathbb{Q}}$ is the dual of $\sigma$ is called the \emph{evaluation map}.  Recall that a variety $Y$ is \emph{semi-projective} if its $\mathbb{K}$-algebra of global functions $A_0 = \Gamma (Y, \mathcal{O}_Y)$ is finitely generated and in addition the structure morphism $Y \to \mathrm{Spec}(A_0)$ is projective; cf.~\cite[Def.~2.1]{AH06}.

\begin{definition}
{\em 
Let $\D$ be a $\sigma$-polyhedral divisor over $(Y, N)$. We say that $\D$ is a 
\emph{proper polyhedral divisor} or a \emph{$p$-divisor}  if $\D(m)$ is semi-ample for any $m\in \sigma^{\vee}$, and $\D(m)$ is big for $m$ in the relative interior of $\sigma^{\vee}$. Note that this condition holds automatically when $\loc(\D)$ is affine.
}
\end{definition}

Let $\D$ be a $p$-divisor over $(Y, N)$. Since $\D(m) + \D(m') \leq \D(m +m')$ for all $m, m'\in\sigma^{\vee}$, the multiplication
in the function field of $Y$ induce the structure of  an $M$-graded $\oo_{\loc(\D)}$-algebra on
$$\A(\D):= \bigoplus_{m\in \sigma^{\vee}\cap M}\oo_{\loc(\D)}(\D(m)).$$
We denote by $\widetilde{X}(\D)$ the relative spectrum of $\A(\D)$ over $Y$ and by $X(\D)$ the spectrum of the ring of global 
sections. Both are naturally equipped with an algebraic $\TT$-action induced by the $M$-grading on $\A(\D)$. 
The following result explains the importance of this construction in the theory of affine $\mathbb{T}$-varieties;  see~\cite[Thms.~3.1 and 3.4; Sect.~6]{AH06}.
 
\begin{theorem}\label{thm:p-div}
The following statements hold:
\begin{itemize}
\item[$(i)$] For any $p$-divisor $\D$ over $(Y, N)$, the quasiprojective scheme  $\widetilde{X}(\mathcal{D})$ and the affine scheme $X(\D)$ are normal $\TT$-varieties. The canonical morphism $\phi\colon \widetilde{X}(\mathcal{D})\rightarrow X(\D)$ is a $\mathbb{T}$-equivariant projective birational contraction. Moreover, $X(\D)$ contains an open subset admitting a geometric quotient that is birational to $Y$. 
\item[$(ii)$] Conversely, any normal affine variety $X$ with faithful $\mathbb{T}$-action is of the form $X(\D)$
for some $p$-divisor $\D$ over some semi-projective normal variety $Y$ with structure morphism $Y \to X /\!/\TT$. 
\end{itemize}

\end{theorem}

\subsubsection{Invariant Weil divisors on $\mathbb{T}$-varieties}
 Let $X  =  X(\D)$ be an affine $\TT$-variety associated with a $p$-divisor $\D$ and set $Y =  \loc(\D)$. 

\begin{definition}{\em 
$\TT$-invariant prime divisors on $X(\D)$ are divided into two sorts: 
\begin{itemize}
\item[$(1)$] the ones whose restriction of the vanishing order to $\kk(X)^{\TT}\simeq \kk(Y)$ is non-trivial, and 
\item[$(2)$] the divisors whose restriction of the vanishing order to $\kk(X)^{\TT}\simeq \kk(Y)$ is trivial. 
\end{itemize}
Divisors of type $(1)$ are called \emph{vertical divisors} while the ones of type $(2)$ are the \emph{horizontal divisors}.}
\end{definition}

For a $\qq$-Cartier divisor $D$ on a normal variety
$W$ and a prime divisor $Z$ on $W$, we say that the restriction $D_{| Z}$ is \emph{big} if for some $r\in \zz_{>0}$
such that $rD$ is integral Cartier, the invertible sheaf $\mathcal{O}_{W}(rD)_{|Z}$ is big. In the sequel, we denote by
$\ver(\D)$ (resp. $\ray(\D)$) the set of pairs $(Z, v)$ such that $Z$ is a prime divisor of $Y$, $v$ is a vertex of $\D_{Z}$
and $\D(m)_{|Z}$ is big for  $m\in \relint(\D_{Z}-v)^{\vee}$ (resp. the set of rays $\rho$ of $\sigma(\D)$ such that $\D(m)$ is big for 
$m\in \relint(\rho^{\perp}\cap \sigma^{\vee})$).
\textcolor{black}{ Here, ${\rm relint}$ stands for the relative interior of the cone, i.e., the interior of the cone in the smallest $\qq$-vector subspace containing it.}
The following proposition describes invariant prime divisors of $\mathbb{T}$-varieties; see, e.g.~\cite[Propositions 4.11, 4.12]{HS10}.

\begin{proposition}
Let $\mathcal{D}$ be a proper polyhedral divisor
on a normal semi-projective variety.
There are bijections:
\begin{itemize}
\item[$(1)$] 
between the pairs $(Z,v)$, 
where $Z$ is a prime divisor of $Y$ and $v$ a vertex of $\mathcal{D}_Z$,
and vertical $\mathbb{T}$-invariant prime divisors of $\widetilde{X}(\mathcal{D})$, and
\item[$(2)$] between rays of 
$\sigma:={\rm tail}(\D)$ 
and the horizontal $\mathbb{T}$-invariant prime divisors
of  $\widetilde{X}(\D)$.
\end{itemize}
\end{proposition}

For $(Z,v)$ as in the proposition,
we denote by $D_{Z,v}$ the corresponding vertical invariant divisor on $\widetilde{X}(\D)$.
On the other hand, 
for $\rho \in \sigma(1)$, 
we denote by $D_\rho$ the corresponding horizontal invariant
divisor on $\widetilde{X}(\D)$.
We have a natural projective
$\mathbb{T}$-equivariant birational morphism
$\phi \colon \widetilde{X}(\D)\rightarrow X(\D)$.
Some of the $\mathbb{T}$-invariant
prime divisors 
of $\widetilde{X}(\D)$ are contracted by $\phi$; the following proposition characterizes those that survive on $X(\mathcal{D})$; see, e.g.,~\cite[\S 6.2]{AIPSV12}.
\begin{proposition}
Let $\mathcal{D}$ be a proper polyhedral divisor
on a normal semi-projective variety.
There are bijections:
\begin{itemize}
\item[$(1)$] between the set of vertical divisors of $X(\D)$ and the set $\ver(\D)$, and
\item[$(2)$] between the set of horizontal divisors of $X(\D)$ and the set $\ray(\D)$. 
\end{itemize}
\end{proposition}

By abuse of notation,
the push-forward of $D_\rho$ (resp $D_{Z,v}$) to $X(\D)$
will be denoted by the same symbol. This means that for $(Z,v)\in \ver(\D)$ and $\rho\in \ray(\D)$, we will denote by $D_{Z, v}$ and $D_{\rho}$ the corresponding prime
$\TT$-divisors on $X = X(\D)$. Explicitly, the principal divisor of a homogeneous rational function $f\otimes \chi^{m}$,
with $f\in \kk(Y)^{\star}$ and $m\in M$, is given by
\begin{equation}\label{eq:div-toric}
    {\rm div}(f\otimes \chi^m) =  \sum_{\rho\in \ray(\D)} \langle m, v_{\rho}\rangle D_{\rho} + \sum_{(Z, v)\in \ver(\D)}\mu(v)({\rm ord}_{Z}(f) + \langle m, v\rangle) D_{Z, v},
\end{equation}
where $v_{\rho}\in N$ is the primitive generator of the ray $\rho$, and $\mu(v)$ is the smallest integer $r\in \zz_{>0}$ such that $rv\in N$, see e.g.~\cite[Prop.~3.14]{PS11}. In these terms, one can express the canonical class 
of the $\TT$-varieties occuring in the above discussion as follows (see e.g.~Theorem 3.21 of \emph{loc.~cit.}).

\begin{proposition} \label{Proposition-CanonicalClassTVar}
With the same notation as before, the canonical class of the $\TT$-variety $\widetilde{X} =  \widetilde{X}(\D)$
is 
$$K_{\widetilde{X}} =  \sum_{(Z, v)} (\mu(v) K_{Y, Z} + \mu(v) -1) D_{Z, v} -  \sum_{\rho\in \sigma(1)}D_{\rho},$$
where $K_{Y}:= \sum_{Z}K_{Y, Z} Z$ is a canonical divisor of $Y$
and $\sigma$ is the tail-cone of $\D$. 
Furthermore, the canonical class of the $\TT$-variety $X =  X(\D)$
is 
$$K_{X} =  \sum_{(Z, v)\in \ver(\D)} (\mu(v) K_{Y, Z} + \mu(v) -1) D_{Z, v} -  \sum_{\rho\in \ray(\D)}D_{\rho}.$$
\end{proposition}

\section{Torus and finite quotients of klt type singularities}
\label{sec:torus+finite}

In this section, we show that the quotient of an affine klt type variety by a torus action or a finite action yields an affine klt type variety.
The latter statement is well-known to the experts, while the former statement is an application 
of the theory of proper polyhedral divisors.

\subsection{Torus quotients of klt spaces}
We first introduce some technical lemmata\textcolor{black}{, the first of which is immediate.}

\begin{lemma}\label{lem:contracted-div}
Let $X$ be an $n$-dimensional normal quasi-projective variety.
Let $D$ be a semiample divisor on $X$.
Let $\phi\colon X\rightarrow X'$ be the morphism induced by $D$.
Then, any prime divisor $Z$ of $X$ such that $D|_Z$ is not big
is \emph{contracted} by $\phi$, \textcolor{black}{i.e., 
the image $\phi(Z)$ has dimension at most $n-2$.} \hfill \qed
\end{lemma}

\begin{lemma}\label{lem:can-bundl}
Let $\psi \colon Y\rightarrow Y_0$ be a projective contraction from a quasiprojective variety $Y$ onto an affine normal variety $Y_0$.
Assume that there exists a boundary $\Delta$ on $Y$ so that 
$(Y,\Delta)$ has klt singularities
and $-(K_Y+\Delta)$ is nef and big over $Y_0$.
Then, $Y_0$ is of klt type.
\end{lemma}
    
\begin{proof}
Since $-(K_Y+\Delta)$ is big and nef over $Y_0$,
then it is semiample over $Y_0$ (e.g., see ~\cite[Cor.~1.3.1]{BCHM10}), and hence semiample.
Consequently, we can find $-(K_Y+\Delta) \sim_{\qq,Y_0} \Gamma\geq 0$
so that 
$(Y,\Delta+\Gamma)$ has klt singularities, cf.~\cite[Lem.~5.17(2)]{KM98}. 
Observe that by construction we have
\[
K_Y+\Delta+\Gamma \sim_{\qq,Y_0} 0. 
\]
Hence, by the canonical bundle formula, e.g., see~\cite[Lem.~1.1]{FG12}, 
we can find a $\mathbb{Q}$-divisor $\Delta_0 \geq 0$ on $Y_0$ such that
\[
K_Y+\Delta+\Gamma \sim_\qq \psi^*(K_{Y_0}+\Delta_0),
\]
and such that $(Y_0,\Delta_0)$ has klt singularities.
We conclude that $Y_0$ is of klt type.
\end{proof}

\begin{lemma}\label{lem:big-restr}
Let $\mathcal{D}$ be a proper polyhedral divisor on $Y$
with tail-cone $\sigma$.
Let $Z\subset Y$ be a prime divisor.
Let $v$ be a vertex of $\mathcal{D}_Z$.
Let $\tau \subseteq \sigma^\vee$ be the cone of all $u\in \sigma^\vee$
so that $\langle u,-\rangle$ minimizes at $v$ in $\mathcal{D}_Z$.
Assume that $\mathcal{D}(u_0)|_Z$ is big for some $u_0\in \tau$.
Then $\mathcal{D}(u)|_Z$ is big for every $u\in {\rm relint}(\tau)$.
\end{lemma}

\begin{proof}
Let $u\in {\rm relint}(\tau)$.
For $m>0$ large enough, we can write $mu = u_0+v_0$, where $v_0 \in {\rm relint}(\tau)$.
Note that the defining property of $\tau$ implies that
\[
{\rm coeff}_Z(\mathcal{D}(mu)) =
{\rm coeff}_Z(\mathcal{D}(u_0) + \mathcal{D}(v_0)).
\] 
Hence, since $\mathcal{D}(mu)\geq \mathcal{D}(u_0)+\mathcal{D}(v_0)$, 
we can write 
\[
\mathcal{D}(mu) = \mathcal{D}(u_0) + \mathcal{D}(v_0) + F,
\] 
where $F$ is an effective divisor that does not contain $Z$ in its support.
Hence, we have that 
\[ 
m\mathcal{D}(u)|_Z = \mathcal{D}(u_0)|_Z +\mathcal{D}(v_0)|_Z + F|_Z.
\]
Note that $\mathcal{D}(u_0)|_Z$ is big by assumption,
$\mathcal{D}(v_0)|_Z$ is pseudo-effective, being the restriction
of a semiample divisor,
and $F|_Z$ is effective.
In particular,
$\mathcal{D}(u_0)|_Z$ is big, 
and $\mathcal{D}(v_0)|_Z+F|_Z$ is pseudo-effective.
Thus, we conclude that $\mathcal{D}(u)|_Z$ is big.
\end{proof}

With these preparations at hand,  we are now in the position to prove a version of the main result for the special case of torus actions.
The following proposition is partially motivated by~\cite{LS13}.

\begin{proposition}
\label{prop:torus-quotient-klt}
Let $X$ be an affine $\mathbb{T}$-variety.
Let $x\in X$ be a fixed point for the torus action.
Assume that $X$ has klt singularities.
Then,  the quotient $X/\!/\mathbb{T}$ contains an open affine subset $[x] \in U \subseteq X/\!/\mathbb{T}$ that contains the image $[x]$ of $x$ and is of klt type.
\end{proposition}

\begin{proof}
By Theorem~\ref{thm:p-div}, we have a $\mathbb{T}$-equivariant isomorphism $X\cong X(\mathcal{D})$, for some p-divisor on a normal quasi-projective variety $Y$ that is projective over $X/\!/\mathbb{T}$.
Replacing $Y$ with a log resolution of $(Y,\supp(\mathcal{D}))$ and pulling back the p-divisor $\mathcal{D}$ to the log resolution, we may assume that $(Y,\supp(\mathcal{D}))$ is log smooth, cf.~\cite[Lem.~2.1]{LS13}. 
We write $\phi\colon \widetilde{X}(\D)\rightarrow X(\D)$ for the canonical projective equivariant birational morphism.
By the klt condition, we can write
\[
K_{\widetilde{X}(\mathcal{D})} + E = \phi^*(K_{X(\mathcal{D})}).
\]
Here, $E$ is a $\qq$-divisor with coefficients 
strictly less than $1$.
As $x$ is a fixed point, according to Lemma~\ref{lem:shrinking-around-fixed-point}, we may shrink $X(\mathcal{D})$ around $x$, i.e., pass to a $\mathbb{T}$-invariant affine open neighborhood of $x$ on which we may write
\[
r(K_{X(\mathcal{D})}) = {\rm div}(f\otimes \chi^m),
\]
for some suitable $r \in \mathbb{N}$, $f\in \kk(Y)^*$, and $m\in M$. Without loss of generality, we will assume that the neighborhood is equal to the whole of $X$.

Let $\rho$ be any ray of $\sigma$, where $\sigma$ is the tail-cone of $\mathcal{D}$.
Then, by Proposition~\ref{Proposition-CanonicalClassTVar} and Equation~\eqref{eq:div-toric}, we conclude that 
\[
0< a_{D_\rho}(X(\mathcal{D}))= 
1+{\rm coeff}_{D_\rho}
(K_{\widetilde{X}(\mathcal{D})}) - 
{\rm coeff}_{D_\rho}({\rm div}(f\otimes \chi^m)) =
\langle -m, v_\rho \rangle ,
\]
where $v_\rho$ is the lattice generator of $\rho$.
We deduce that $-m\in {\rm relint}(\sigma^\vee)$. 

Let $(Z,v)$ so that the value 
$\langle -m, v\rangle$ is minimal in $\mathcal{D}_Z$.
By Proposition~\ref{Proposition-CanonicalClassTVar} and
Equation~\eqref{eq:div-toric}, we obtain that the equality
\begin{equation}\label{eq:xtilda} 
\sum_{(Z,v)}(\mu(v)K_{Y,Z}+\mu(v)-1)D_{Z,v}
-\sum_{\rho} D_\rho + E = \\
\sum_{(Z,v)}\mu(v)({\rm ord}_Z(f) + \langle m,v\rangle ) D_{Z,v} + \sum_{\rho} \langle m, v_\rho\rangle D_\rho
\end{equation} 
holds in $\widetilde{X}(\mathcal{D})$.
We write $E_{Z,v}$ 
for the coefficients of $E$ at $D_{Z,v}$.
Analogously, we write $H_Z$ for the coefficient of 
$H=\frac{1}{r}\phi^*({\rm div}(f))$ at $Z$.
For each $(Z,v)$, considering the coefficient of both sides of equality~\eqref{eq:xtilda} at $D_{Z,v}$,
we obtain that
\[
\mu(v)K_{Y,Z}+\mu(v)-1 +{\rm coeff}_{D_{Z,v}}(E) = \mu(v)({\rm ord}_Z(f)+\langle m, v\rangle) 
\] 
holds for every $Z\subset Y$.
Dividing by $\mu(v)$, we deduce that 
\[
K_{Y,Z}+\sum_{Z\subset Y} 
\left( 
\frac{\mu(v)-1}{\mu(v)} + \frac{E_{Z,v}}{\mu(v)}
\right) =
H_Z + \langle m, v\rangle,
\] 
holds in $Y$.
Note that $H_Z\sim_\qq 0$.
Hence, if we set 
\[
B_Y = \sum_{Z\subset Y}
\left( 
\frac{\mu(v)-1}{\mu(v)} + \frac{E_{Z,v}}{\mu(v)}
\right)Z,
\]
we get that 
\begin{equation}\label{eq:nef-model}
K_{Y} + B_Y \sim_{\qq, X/\!/\mathbb{T}} -\mathcal{D}(-m),
\end{equation}
where $(Y,B_Y)$ is log smooth and
the coefficients of $B_Y$ are less than $1$.
Indeed, as noted before, the coefficients $E_{Z,v}$ of $E$ are less than one by the klt assumption.
Note that $(Y,B_Y)$ is only a sub-pair, as
$B_Y$ may have negative coefficients.
As $-m$ is in the relative interior of $\sigma^{\vee}$ we deduce from equality~\eqref{eq:nef-model} that $-(K_Y+B_Y)$ is semiample and big over $X/\!/\mathbb{T}$.

Let $Y_0$ be the ample model of $-(K_Y+B_Y)$ over $X/\!/\mathbb{T}$.
Let $\phi_0\colon Y\rightarrow Y_0$ be the induced projective birational morphism.
Then, we have that 
\begin{equation}\label{eq:ample-model}
\phi_0^*(K_{Y_0}+B_{Y_0})=K_Y+B_Y,
\end{equation}
for certain $\qq$-divisor $B_{Y_0}$ on $Y_0$. In a next step, we are going to show that $B_{Y_0}$ is effective. Indeed, if $\mathcal{D}(-m)|_Z$ is big, 
then by Lemma~\ref{lem:big-restr}
we conclude that $D_{Z,v}$ is not contracted
by the projective birational map $\phi: \widetilde{X}(\mathcal{D})\rightarrow X(\mathcal{D})$.
Conversely, if $E_{Z,v}$ is a negative number, 
then $\mathcal{D}(-m)|_Z$ is not big.
In particular, we conclude that if $B_Y$ has a negative coefficient 
at the prime divisor $Z$,
then $\mathcal{D}(-m)|_Z$ is not big,
so $-(K_Y+B_Y)|_Z$ is not big.
By Lemma~\ref{lem:contracted-div},
it follows that the exceptional locus
of the projective birational morphism 
$\phi_0$
contains all the prime divisors
at which $B_Y$ has negative coefficients, and hence that $B_{Y_0}$ is an effective divisor.
From Equality~\eqref{eq:ample-model} and the construction of $B_Y$, see the line after Equation~\eqref{eq:nef-model}, we conclude that 
$(Y_0,B_{Y_0})$ is a log pair with klt singularities.
Thus, we conclude that
$(Y_0,B_{Y_0})$ has klt singularities
and $-(K_{Y_0}+B_{Y_0})$ is ample over $X/\!/\mathbb{T}$.
By Lemma~\ref{lem:can-bundl}, we conclude that
$X/\!/\mathbb{T}$ is of klt type, as claimed.
\end{proof}

\subsection{Finite quotients of klt type singularities}
Now, we turn to prove that the quotient of an affine variety of klt type by a finite group
is again an affine variety of klt type; this is well-known to the experts, see~\cite[Prop.~2.11]{Mor21}.
We give a short argument for the sake of completeness.

\begin{proposition}\label{prop:finite-quotient}
Let $X$ be an affine variety of klt type.
Let $G$ be a finite group acting on $X$.
Then, the affine variety $X/\!/G$ is of klt type as well.
\end{proposition}

\begin{proof}
We denote the (finite) quotient morphism 
by $\phi\colon X\rightarrow X/\!/G$, and recall that $X/\!/G$ is a normal affine variety. Since $X$ is of klt type, we can find a boundary $B$ so that the pair $(X,B)$ is klt.
We define 
\[
B^G:=\frac{\sum_{g\in G}g^*B}{|G|}.
\]
Note that the divisor $K_X+B^G$ is $\qq$-Cartier and the pair $(X,B^G)$
has klt singularities.
Indeed, for every prime exceptional divisor $E$ over $X$ and $\pi \colon Y \to X$ extracting $E$, due to linearity of $\mathrm{coeff}_E$ and the pullback of $\mathbb{Q}$-Cartier divisors, we have that 
\begin{align*}
    a_E(X,B^G) &=1+ \mathrm{coeff}_E(K_Y-\pi^*(K_X+B^G)) = 1+ \mathrm{coeff}_E(K_Y -\pi^*(K_X)) -\mathrm{coeff}_E(\pi^*(B^G)) \\
    &= \frac{|G|(1+ \mathrm{coeff}_E(K_Y -\pi^*(K_X))) - \sum_{g \in G} \mathrm{coeff}_E(\pi^*(g^* B))}{|G|} \\
    &= \sum_{g\in G} \frac{a_E(X,g^*B)}{|G|} >0.
\end{align*}
Replacing $B$ with $B^G$, we may assume that
the pair $(X,B)$ is klt and
$B$ is a $G$-invariant divisor. By the Riemann-Hurwitz formula, there exists an effective boundary divisor $B_{X/\!/G}$ for which
\begin{equation}\label{eq:defeq1}
K_X+B = \phi^*(K_{X/\!/G}+B_{X/\!/G}).
\end{equation}
In order to prove our claim, it suffices to show that the pair $({X/\!/G},B_{X/\!/G})$ has Kawamata log terminal singularities. Using \eqref{eq:defeq1}, this however can be deduced from \cite[Prop.~5.20]{KM98}.
\end{proof}

\section{Being of klt type is an \'etale property}
\label{sec:etale}

In this section, we prove that the klt type condition is an \'etale property.
Further on, we use this statement together with Luna's \'etale slice theorem to prove a reduction of our problem to the case of neighborhoods of points fixed by a  reductive group action. 

\begin{proposition}\label{prop:klt-type-etale}
 Let $X$ be a normal quasi-projective algebraic variety which is \'etale locally of klt type.
 Then, there exists a boundary $B$ on $X$ for which $(X,B)$ is klt; i.e., $X$ is of klt type.
\end{proposition}

\begin{proof}
First, we show that $X$ is Zariski locally of klt type if and only if $X$ is of klt type. 
For this, we use the definition of multiplier ideal sheaves $\mathcal{J}(X)$ for normal - not necessarily $\mathbb{Q}$-Gorenstein varieties - from~\cite[Definition~4.8]{dFH09}. By~\cite[Corollary~7.6]{dFH09} combined with ~\cite[Theorem~1.2]{dFH09}, $X$ is of klt type if and only if the coherent ideal sheaf $\mathcal{J}(X)$ equals the structure sheaf $\mathcal{O}_X$. Since this equality is clearly local in the Zariski topology, we get the desired equivalence.

\begin{comment}
Let $U_1,\dots,U_k$ be an affine open cover so that for each $i$, there exists a boundary $\Delta_i$ for which $(U_i,\Delta_i)$ is klt. We can pick $m$ large enough so that $m\Delta_i \in |-mK_X|_{U_i}|$ for each $i$. Next, we pick $H$ ample so that 
$\mathcal{O}(-mK_X+mH)$ is globally generated.
We claim that for $D\in |-mK_X+mH|$ general enough, the pair
$(X,\frac{D}{m})$ is klt.
Indeed, \cite[Cor.~2.33]{KM98} together with klt-ness of the pair $(U_i, \Delta_i)$ obtained from a special element in the same basepoint-free linear system on $U_i$ implies that $(U_i, \frac{D}{m}|_{U_i})$ is klt for any choice of $i$. We conclude that, if $X$ is Zariski locally of klt type, then it is of klt type.
\end{comment}

The above first step allows us to assume for the remainder of the proof that $X$ itself admits a surjective \'etale morphism from a klt type variety. \textcolor{black}{So, let $u_0\colon U_0\rightarrow X$ be an \'etale
cover such that there exists an effective divisor $\Delta_{U_0}$ on $U_0$ so
that $(U_0,\Delta_{U_0})$ has klt singularities.} We have a diagram as follows:
\[
\xymatrix{
U_0 \ar[dr]_-{u_0} \ar@{^{(}->}[r] & U\ar[d]^-{u} & Y\ar[l]^-{h} \ar[ld]^-{f} \\
& X, &
}
\]
\textcolor{black}{where $u\colon U\rightarrow X$ is the finite morphism obtained by applying Grothendieck's form of Zariski's Main Theorem \cite[p.~209]{MumfordRedBook},}
and
$f\colon Y\rightarrow X$ is the Galois closure of $U\rightarrow X$.
In particular, both $h$ and $f$ are given by the quotient of finite groups acting on $Y$.
Let $G$ be the finite group acting on $Y$ so that
$X=Y/G$.
Let $R_f$ be the ramification divisor of $f$
and let $R_h$ be the ramification divisor of $h$.
We have that $R_f\geq R_h$ and $R_f$ is $G$-invariant.
Then, for every $g\in G$, we have that $R_f\geq g^*R_h$. In particular, for every $g\in G$ and $U\subset Y$ open, we have 
\begin{equation}\label{eff}
R_f|_U  \geq g^*R_h|_U.
\end{equation} 
Let $U_{0,Y}$ be the pre-image of $U_0$ in $Y$ with respect to $h$.
Let $(U_{0,Y},\Delta_Y - R_h)$ be the log pull-back
of $(U_0,\Delta_{U_0})$ to $U_{0,Y}$, i.e., 
we have that 
\[
h^*(K_{U_0}+\Delta_{U_0})=
K_{U_{0,Y}}+\Delta_Y-R_h.
\] 
Note hat the sub-pair
\begin{equation}\label{eq:sub-klt-pair}
(U_{0,Y},\Delta_Y - R_h)
\end{equation}
is sub-klt, as it is the pull-back of a klt pair (see, e.g.,~\cite[Prop.~5.20]{KM98}).
Acting with $G$, we obtain an open affine cover
$U_1,\dots,U_k$ of $Y$ such that $(U_i, \Delta_i - g^*R_h|_{U_i})$
is sub-klt for each $i\in \{1,\dots, k\}$.
Indeed, these pairs are sub-klt as the pair~\eqref{eq:sub-klt-pair} is sub-klt.
In the previous sub-pair, we have 
$U_i=g(U_{0,Y})$ and $\Delta_i=g^*\Delta_Y$ for some $g\in G$.
Set
\[
\Delta'_i = \Delta_i + (R_f|_{U_i} - g^*R_h|_{U_i}) \geq 0.
\]
Here, the fact that $\Delta'_i\geq 0$ follows from Equation~\eqref{eff}.
Then, we obtain that $(U_i,\Delta_i' - R_f|_{U_i})$
is a sub-klt pair for each $i\in \{1,\dots, k\}$. 

There exists a natural number $m$ such that for each $i$ we have $\Delta'_i \in |(-mK_Y + m R_f)|_{U_i}|$.
Next, pick $H$ ample on $Y$
such that
$\mathcal{O}(-mK_Y +mR_f+mH)$ is globally generated. \textcolor{black}{Then, by the considerations of the previous paragraph we may assume that for each $i \in \{1, \dots, k\}$ there exists an element $D_i \in |-mK_Y +mR_f+mH|$ such that the sub-pair 
\begin{equation}\label{subpair1}
    \left(U_i, (\frac{1}{m}D_i - R_f)|_{U_i} \right)
\end{equation}
is sub-klt. Applying \cite[Cor.~2.33]{KM98} to  the sub-pair $(U_i, -R_f)$ and comparing discrepancies over $U_i$ with the discrepancies of the sub-pair \eqref{subpair1}, we conclude that for a general element $D\in |-mK_Y + mR_f + mH|$, the sub-pair $(Y, \frac{1}{m}D - R_f)$ is sub-klt. From this, we infer that also the $G$-closure
\[
\left( 
Y, \sum_{g\in G} \frac{g^*D}{|G|m} - R_f
\right) 
\] of the sub-pair just considered is a $G$-invariant sub-klt pair, cf.~the first paragraph of the proof of Proposition~\ref{prop:finite-quotient}.} We set 
\[
\Gamma_Y:= \sum_{g\in G}\frac{g^*D}{|G|m} - R_f.
\]

By the Riemann-Hurwitz formula, there exists a boundary divisor $\Gamma_X$ on $X$ with
\[
K_Y + \Gamma_Y = f^*(K_X+\Gamma_X).
\]
In particular, $(X,\Gamma_X)$ is a log pair, \textcolor{black}{i.e., $\Gamma_X$ is effective, since the negative part of $\Gamma_Y$ was $-R_f$ with $R_f$ the ramification divisor of $f$.}
Moreover, \cite[Prop.~5.20]{KM98} implies that the pair $(X,\Gamma_X)$ is Kawamata log terminal. We conclude that the variety $X$ is of klt type, as claimed.
\end{proof}

\begin{remark}
{\em The statement of Proposition~\ref{prop:klt-type-etale} is the same as that of \cite[Lem.~8.5]{CS21}. The main point of the proof given above is to fill in the argument for the first claim of the proof given in \emph{loc.~cit.}. While this uses ideas similar to those presented in the main part of the proof of \emph{loc.~cit.}, it also leads to a reorganization of the argument. }
\end{remark}

\section{Lifting group actions to the iterated Cox ring}
\label{sec:iteration}

In this section, we first define the iteration of (aff-local) Cox rings and then show how actions of \emph{connected reductive groups} can be lifted to the spectra of iterated Cox rings.

\subsection{Iteration of Cox rings}

Iteration of Cox rings has been introduced in~\cite{ABHW18}. 
The underlying idea is very simple: In the classical case of the usual Cox ring $R(X)$ of a variety $X$ graded by $\Cl(X)$, one considers the Cox ring $R(\overline{X})$ of $\overline{X}:=\Spec R(X)$ and iterates this procedure. 

The natural question arises whether this iteration comes to a halt, i.e., ---when we assume finite generation of Cox rings at every step--- whether the class group and thus the Cox ring becomes trivial after finitely many steps.
This holds true in many cases: trivially for toric varieties, for ADE surface singularities, and more generally for klt quasi-cones with a torus action of complexity one~\cite{ABHW18, Bra19}, and, due to~\cite{BM21}, for the $\Cl(X,x)$-graded local Cox rings $R(X_x)$ of klt singularities.

In the setting that focusses on aff-local Cox rings in the sense of Definition~\ref{def:aff-loc-Cox-ring} 
%and its spectrum $X'$, we can neither guarantee the finite generation of $\Cl(X')$ nor is this the right question to consider. 
we want the iteration to reflect what happens above the singularity $x$. So we have to iterate with respect to the local class group at the \emph{distinguished point} $x' \in X'$ lying over $x$, c.f.~Proposition~\ref{prop:aff-loc-Cox-ring} \textcolor{black}{(4). Item (5)} of this proposition tells us that the gr-localization of $X'$ at $x'$ yields the local Cox ring. Moreover, by~\cite[Lem.~2.10]{BM21}, the local class group at $x' \in X'$ equals the class group of the gr-localization. Thus, the iteration of aff-local Cox rings defined in this way reflects the iteration of local Cox rings $R_{X_x}$.

Motivated by these considerations we make the following definition; we continue to use the notation of Definition~\ref{def:aff-loc-Cox-ring}.
\begin{definition}
{\em 
Let $X$ be an affine variety of klt type
and $x\in X$.
We denote $R^{(1)}:=R_{X,x}$, $X^{(1)}:=X'$, and $x^{(1)}:=x'$. For $n$ in $\nn$, we choose $K,L \subseteq \WDiv(X^{(n-1)})$ and $\chi \colon L \to \kk(X^{(n-1)})^*$ as in Definition~\ref{def:aff-loc-Cox-ring} and define recursively 
$$
R^{(n)}:= R_{X^{(n-1)},x^{(n-1)}},
\quad
X^{(n)}:=\Spec R^{(n)},
$$
and $x^{(n)} \in X^{(n)}$ to be the unique invariant point in the preimage of $x^{(n-1)}$. 
We call the (possibly infinite) chain of pointed affine varieties 
$$
\cdots \to (X^{(3)},x^{(3)}) \to (X^{(2)},x^{(2)}) \to (X^{(1)},x^{(1)}) \to (X,x)
$$
the \emph{iteration of aff-local Cox rings} of $x \in X$.
}
\end{definition}

\begin{proposition}
\label{prop:Cox-ring-iteration}
Let $X$ be an affine variety of klt type
and $x\in X$.
Then, the following statements hold:
\begin{enumerate}
    \item For every $n \in \nn$, the $\kk$-algebra $R^{(n)}$ is finitely generated, and $X^{(n)}$ is of klt type. 
    \item There exists an $n_0 \in \nn$, such that the local class group $\Cl(X^{(n_0)},x^{(n_0)})$ is trivial; i.e., $x^{(n_0)} \in X^{(n_0)}$ is factorial. If $n_0$ is minimal with this property, we say that \emph{the iteration of Cox rings of $x \in X$ stops at $x^{(n_0)} \in X^{(n_0)}$}.
    \item For $1 \leq n \leq n_0$, the map $\varphi^{(n)}\colon X^{(n)} \to X^{(n-1)}$ is a good quotient by the quasi-torus 
    $$
    H^{(n)}:=\Spec \kk[\Cl(X^{(n-1)},x^{(n-1)})].
    $$
    \item For $1 \leq n \leq n_0$, the composed map
    $$
    \psi^{(n)} := \varphi^{(1)} \circ \ldots \circ \varphi^{(n)} \colon (X^{(n)},x^{(n)}) \to (X,x)
    $$
    down to $X$ is a GIT-quotient by a solvable reductive group $S^{(n)}$.
\end{enumerate}
\end{proposition}

\begin{proof}%[Proof of Proposition~\ref{prop:Cox-ring-iteration}]
Item (1) follows directly from Items (1) and (2) of Proposition~\ref{prop:aff-loc-Cox-ring}, since (recursively) $X^{(n-1)}$  fulfills the requirements of Definition~\ref{def:aff-loc-Cox-ring} for every $n \in \nn$. 

Item (2) follows because the iteration of aff-local Cox rings reflects the iteration of local Cox rings as defined in~\cite{BM21}. 
In particular,  denoting by $X^{(1)}_{x^{(1)}}$ and $X^{(1)}_{(x^{(1)})}$ the localization   and gr-localization of  $X^{(1)}$ at $x^{(1)}$ respectively, Item (4) of Proposition~\ref{prop:aff-loc-Cox-ring} tells us that $X^{(1)}_{(x^{(1)})}$ equals the spectrum of the local Cox ring $R_{X_x}$, and moreover, $X^{(1)}_{x^{(1)}} = (X^{(1)}_{(x^{(1)})})_{x^{(1)}}$. By \cite[Lem. 2.10]{BM21}, the maps $X^{(1)}_{x^{(1)}} \hookrightarrow X^{(1)}_{(x^{(1)})} \hookrightarrow X^{(1)}$ induce isomorphisms of class groups $\Cl(X^{(1)}_{x^{(1)}}) \cong \Cl(X^{(1)}_{(x^{(1)})}) \cong \Cl(X^{(1)},x^{(1)})$. This holds analogously for any $n \in \mathbb{N}$ and so, the iteration of aff-local Cox rings stops when the iteration of local Cox rings stops. This happens after finitely many steps with a factorial point $x^{(n_0)} \in X^{(n_0)}$ by \cite[Thm. 4.17]{BM21}, which in turn relies on the finiteness of local fundamental groups of klt singularities~\cite[Thm. 1]{Bra20}.

Item (3) is clear by definition, cf.~also the proof of Proposition~\ref{prop:aff-loc-Cox-ring}. 

The proof of Item (4) is analogous to that of~\cite[Cor. 4.12]{BM21} (see also~\cite[Thm.~1.6]{ABHW18}).
\end{proof}

\subsection{Lifting of group actions}

In the last subsection, we defined the iteration of aff-local Cox rings. 
Now, we turn to the fact that certain group actions can be lifted to the spectra of Cox rings.
This property turns out to be the key for transferring statements about quotients of klt varieties $X$ (with trivial boundary) or even factorial ones to the case of a klt pair $(X,\Delta)$ (with nontrivial boundary). The main reference is~\cite[Sec 4.2.3]{ADHL15} in the classical case and~\cite[Sec.~4.1]{BM21} for generalized Cox rings. 

First, we recall here~\cite[Def.~4.2.3.1]{ADHL15} in our setting.

\begin{definition}
\label{def:lifting-to-Cox}
{\em 
Let $X$ as well as $K,L \subseteq \WDiv(X)$ and $\chi \colon L \to \kk(X)^*$ be as in Definition~\ref{def:aff-loc-Cox-ring},  with associated Cox space $\varphi \colon X' \to X$. Given an algebraic group action $\mu \colon G \times X \to X$, we say that a finite epimorphism $\varepsilon\colon G' \to G$ and a group action $\mu'\colon G' \times X' \to X'$ \emph{lift} the $G$-action to the spectrum $X'$ of the aff-local Cox ring, if
\begin{enumerate}
    \item the actions of $G'$ and $H':=\Spec \kk[\Cl(X,x)]$ on $X'$ commute,
    \item the epimorphism $\epsilon$ satisfies $\varphi(g'\cdot y) = \varepsilon(g') \cdot \varphi(y)$ for all $g' \in G'$ and $y \in X'$.
\end{enumerate}
}
\end{definition}

\begin{remark}
{\em 
\textcolor{black}{
For an action of a  connected reductive group $G$ on $X$,  from~\cite[Thm 4.2.3.2]{ADHL15} we see there is always a lifting of the $G$-action to $X'$. In fact, the assumption in \emph{loc.~cit.} that $X$ should have only constant invertible functions is only needed in order for the Cox ring to be well defined in that setting.  In particular, the proof uses only the existence of $G$-linearizations of sheaves $\mathcal{A}$ of graded algebras on $X$, cf.~\cite[Def.~4.2.3.3]{ADHL15}, which in turn only requires $X$ to be a pre-variety endowed with the action of a connected algebraic group.}

If $G$ is simply connected, then the epimorphism $\varepsilon$ is just the identity~\cite[Thm 4.2.3.2~(iii)]{ADHL15}. If $G$ is a torus, then $G'=G$ and $\varepsilon$ is of the form $g \mapsto g^b$ for some $ b \in \zz^+$ by \cite[Thm 4.2.3.2~(iv)]{ADHL15}. If $G$ is semisimple, but not simply connected, e.g. a special orthogonal group, then $\varepsilon$ may come from a nontrivial covering group of $G$. In particular, this covering group is semisimple and we can always assume it to be simply connected by the construction in~\cite[Sec.~4.2.3]{ADHL15}.
}
\end{remark}

The following proposition shows that not only can this lifting be iterated so that there is a lifting of $G$ to the iterated Cox space $X^{(n)}$ for any $n \in \nn$, i.e., a finite epimorphism $\varepsilon^{(n)}\colon G^{(n)} \to G$ and an action $\mu^{(n)}\colon G^{(n)} \times X^{(n)}$ such that Items (1) and (2) of Definition~\ref{def:lifting-to-Cox} are fulfilled with $H'$ and $\varphi$ replaced by  $H^{(n)}$ and $\varphi^{(n)}$ from Proposition~\ref{prop:Cox-ring-iteration} (3) respectively, but, in fact, the action of $G^{(n)}$ \emph{commutes} with the action of $S^{(n)}$ and \emph{fixes} $x^{(n)}$ if the action of $G$ fixes $x$. 

\begin{proposition}
\label{prop:lift-iteration-Cox-ring}
Let $X$ be an affine  variety of klt type as in Definition~\ref{def:aff-loc-Cox-ring} and $x\in X$.
Assume that a connected reductive group $G$ acts on $X$.
Then, for every $n \in \nn$, the following hold:
\begin{enumerate}
    \item There is a finite epimorphism $\varepsilon^{(n)}\colon G^{(n)} \to G$ and an action $\mu^{(n)}\colon G^{(n)} \times X^{(n)} \to X^{(n)}$.
    \item The action of $G^{(n)}$ on $X^{(n)}$ commutes with the action of $H^{(n)}=\Spec \kk[\Cl(X^{(n-1)},x^{(n-1)})]$. In particular, there is an algebraic action of $H^{(n)}$ on $Y^{(n)}:=X^{(n)}/\!/G^{(n)}$, and we have an identity
    $$
    Y^{(n)}/\!/H^{(n)} \cong  Y^{(n-1)}.
    $$
    \item The composed map $\psi^{(n)} := \varphi^{(1)} \circ \ldots \circ \varphi^{(n)}$ and the epimorphism $\varepsilon^n \colon G^{(n)} \to G^{(n-1)} $ satisfy 
    $$
    \psi^{(n)}(g\cdot y) = \varepsilon(g) \cdot \psi^{(n)}(y)
    $$
    for all $g \in G^{(n)}$ and $y \in X^{(n)}$.
    \item If the action of $G$ on $X$ fixes $x$, then the action of $G^{(n)}$ on $X^{(n)}$ fixes $x^{(n)}$.
    \item If $G$ is semisimple, we can take $G^{(n)}$ to be the universal cover of $G$ for every $n \in \mathbb{N}$.
    \item We have a commutative diagram of quotients:
    
    $$
 \xymatrix@R=40pt@C=60pt{
    X^{(n)} \ar[r]^{/\!/H^{(n)}} \ar[d]^{/\!/G^{(n)}} \ar@/^3.0pc/[rrrr]^{/\!/S^{(n)}}
    &
    X^{(n-1)} \ar[r]^{/\!/H^{(n-1)}} \ar[d]^{/\!/G^{(n-1)}} 
    &
    \cdots \ar[r]^{/\!/H^{(2)}}
    &
    X^{(1)} \ar[r]^{/\!/H^{(1)}} \ar[d]^{/\!/G^{(1)}}
    &
    X \ar[d]^{/\!/G}
    \\
    Y^{(n)} \ar[r]^{/\!/H^{(n)}}
    &
    Y^{(n-1)} \ar[r]^{/\!/H^{(n-1)}}
    &
    \cdots
    \ar[r]^{/\!/H^{(2)}}
    &
    Y^{(1)} \ar[r]^-{/\!/H^{(1)}}
    &
    Y^{(0)} \cong X /\!/G \;\;\;\;\;\;\;\;\;
    }
$$

\end{enumerate}
\end{proposition}

\begin{proof}
We start by observing that $X^{(n)} \to X^{(n-1)}$ fulfills the requirements for a \emph{quotient presentation} in the sense of~\cite[Def.~4.2.1.1]{ADHL15}, except possibly for the assumption on  invertible global regular functions. But as we explained above, this is not needed in the proof of~\cite[Thm.~4.2.3.2]{ADHL15}. So, the first and the third item of the proposition follow directly by recursively applying~\cite[Thm.~4.2.3.2]{ADHL15}, yielding a chain of (possibly trivial) covers
$$
G^{(n)} \xrightarrow{\varepsilon^{n}} G^{(n-1)} \xrightarrow{\varepsilon^{n-1}} \cdots \xrightarrow{\varepsilon^{2}} G^{(1)} \xrightarrow{\varepsilon^{1}} G,
$$
such that $\varepsilon^{(n)}:= \varepsilon^{1} \circ \cdots \circ \varepsilon^{n} \colon G^{(n)} \to G$.

We come to Item (2). Here, the fact that the action of  $G$ on $X^{(n)}$ commutes with the action of $H^{(n)}$ again follows from~\cite[Thm 4.2.3.2]{ADHL15}. 
\begin{comment}
For commuting with $S^{(n)}$, we proceed by induction on $n$. For $n=1$, we have $S^{(1)}=H^{(1)}$. So let's assume that the action of $G^{(n-1)}$ on $X^{(n-1)}$ commutes with the action of $S^{(n-1)}$.
By~\cite[Lem.~4.7]{BM21}, we have a short exact sequence
$$
1 \to H^{(n)} \to \Aut^{H^{(n)}}(X^{(n)}) \xrightarrow{\pi} \Aut(X^{(n-1)}) \to 1,
$$
where $\Aut^{H^{(n)}}(X^{(n)})$ is the normalizer of $H^{(n)}$ in   $\Aut(X^{(n)})$.
The group $S^{(n)}$ is an extension of $S^{(n-1)}$ by $H^{(n)}$ given by the preimage $\pi^{-1}(S^{(n-1)})$. On the other hand, $G^{(n)}$ is a subgroup of $\Aut^{H^{(n)}}(X^{(n)})$ that even lies in the centralizer of $H^{(n)}$. So, elements $s \in S^{(n)}$ and $g \in G^{(n)} \subseteq \Aut^{H^{(n)}}(X^{(n)})$ commute as sson as the residue class $[s][g]-[g][s]=[sg-gs] \in \Aut(X^{(n-1)})$ vanishes. But $[s] \in S^{(n-1)}$ and $[g] \in G^{(n-1)}$ commute in $\Aut(X^{(n-1)})$ by our induction hypothesis.
\end{comment}
The action of $H^{(n)}$ on $Y^{(n)}$ is induced by the action of $H^{(n)}$ on the ring $\kk[X^{(n)}]^{G^{(n)}}$ of $G^{(n)}$-invariants. On the other hand, by construction of the lifting, cf.~Definition~\ref{def:lifting-to-Cox}, the induced action of $G^{(n)}$ on the quotient $X^{(n-1)}=X^{(n)}/\!/H^{(n)}$ is via the epimorphism $\epsilon^{(n)}$ from $G^{(n)}$ to $G$, so that the right hand side is equal to $X^{(n-1)}/\!/G^{(n)}$.
This establishes the claim made in the second item.

For Item (4), we note that by Definition~\ref{def:lifting-to-Cox} (ii), for $g \in G^{(n)}$ the image of $g \cdot x^{(n)}$ in $X^{(n-1)}$ is $x^{(n-1)}$. This means that if  $g \cdot x^{(n)} \neq x^{(n)}$, since $x^{(n)}$ is the unique $H^{(n)}$-fixed point in the preimage of $x^{(n-1)}$, there exists an $h \in H^{(n)}$ such that $h\cdot(g \cdot x^{(n)}) \neq g \cdot x^{(n)}=g \cdot (h \cdot x^{(n)})$. This however contradicts the fact that $G^{(n)}$ lies in the centralizer of $H^{(n)}$. So $G^{(n)}$ needs to fix  $x^{(n)}$, which proves Item (4).

Item (5) follows from the construction of the lifting in~\cite[Pf. of Thm.~4.2.3.2, Lem. 4.2.3.9]{ADHL15}.

Item (6) finally follows from recursively applying Item (2).
\end{proof}

\section{Proofs of the main theorems and their corollaries} 
In the first subsection, we essentially deal with the case of semisimple groups, before giving the proof of the main theorems in the second subsection.
\subsection{Factoriality of quotients by semisimple groups}
It is well known that, 
if $A$ is a unique factorization domain
and $G$ is a connected semisimple group acting on it, 
then $A^G$ is a unique factorization domain as well (see, e.g.,~\cite[Thm.~II.3.17]{AGIV}).
To prove our main theorem, 
we will show a local version of the previous statement.
Indeed, we show that the implication holds locally at a point $x$ in $X$, whenever $X$ is factorial at $x$ and $G$ is a connected semisimple group fixing $x$.

\begin{lemma}\label{lem:semisimple-quotient}
\label{le:factorial-semisimple-quotient}
Let $X$ be an affine variety of klt type such that
\begin{enumerate}
    \item a connected semisimple group $G$ acts on $X$,
    \item there is a $G$-fixed closed point $x \in X$,
    \item the Zariski local ring $\mathcal{O}_{X,x}$ is factorial.
\end{enumerate}
Then there is an affine $G$-stable open neighborhood $x \in U \subseteq X$, such that
for the quotient map $\pi\colon X \to X/\!/G$, the image $\pi(U) \subseteq X/\!/G$ is open, affine, locally factorial, and has canonical singularities.
\end{lemma}

\begin{proof}
First, we note that since $X$ is of klt type, it has rational singularities, and thus, by~\cite{Bou87}, the quotient $X/\!/G$ has rational singularities.
In a second step, by Lemma~\ref{lem:shrinking-around-fixed-point} (i), we can replace $X$ with a locally factorial principal open neighborhood $X_f$ of $x$ and thus assume that $X$ is locally factorial. 

In a third step, we now show that the local ring $\mathcal{O}_{X/\!/G,[x]}$ at the point $  [x]=\pi(x) \in X/\!/G$ is factorial. \textcolor{black}{We apply \cite[Thm.~II.3.1]{AGIV} to see that $G$ transforms each element of $\mathbb{K}[X]^*$ by a character. However, as a connected semisimple group, $G$ has no nontrivial characters, so that the action of $G$ on $\mathbb{K}[X]^*$ is actually trivial.} Let now $W$ be an effective Weil divisor through $[x]$.
 We restrict $W$ to the smooth locus of $X/\!/G$ and  denote by
 $W_X$ the closure in $X$ of the pullback of this restriction.
 By construction, $W_X$ is a $G$-invariant effective Weil divisor.
  Since $X$ is locally factorial by the second step of the proof, $W_X$ is Cartier.

  By Lemma~\ref{lem:shrinking-around-fixed-point} (ii), we obtain a $G$-invariant principal open neighborhood $U'':=X_g$ of $x$, where $W_X$ is principal. Thus we can write $\left.W_X\right|_{U''} = \ddivv(h)$ for some regular function $h$ on $U''$.
  We have already seen above that $G$ acts trivially on $\kk[U'']^*$, so $h$ is $G$-invariant.
   Let $D$ be the principal divisor on $U''/\!/G$ defined by $h$,
 so that  $W_X=\pi^*(D)$ holds on $U''$.
 On the other hand, we have that $W_X=\pi^*(W)$ holds over the smooth locus of
$U''/\!/G$.
We conclude that $\pi^*(D)=\pi^*(W)$ holds on $\pi^{-1}((U''/\!/G)^{\mathrm{sm}})$.
Since the map $\pi$ is surjective, this implies that $D = W$ on
$(U/\!/G)^{\mathrm{sm}}$ (see, e.g.,~\cite[Proposition 5.3]{FM20}).
Since $U''/\!/G$ is normal, this means that $D \sim W$ on $U''/\!/G$. So $W$ is principal on a neighborhood of $[x]$.
\textcolor{black}{Since $W$ was an arbitrary effective Weil divisor through $[x]$, it follows that the local ring $\mathcal{O}_{X/\!/G,[x]}$ is factorial.}

By the same argument as in the second step of the proof, we can find an open affine locally factorial neighborhood $\hat{U}$ of $[x]$. The preimage  $U:=\pi^{-1}(\hat{U})$ will likewise be affine. Moreover,  since $\hat{U}$ has rational singularities as noted at the beginning of the proof, and since local factoriality implies in particular that $K_{\hat U}$ is Cartier, $\hat{U} = \pi(U)$ has canonical singularities by \cite[Cor.~5.24]{KM98}. So all claims are proven.
\end{proof}

\subsection{Proof of the main results}\label{subsect:proof_of_main_results}

As a first step, we will address the case of torus actions and generalize Proposition~\ref{prop:torus-quotient-klt} from klt to klt type singularities.

\begin{proposition}\label{prop:torus-quotient}
Let $X$ be an affine $\mathbb{T}$-variety.
Let $x\in X$ be a fixed point for the torus action.
Assume that $X$ is of klt type. 
Then, the quotient $X/\!/\mathbb{T}$ contains an open affine subset $[x] \in U \subseteq X/\!/\mathbb{T}$ that contains the image $[x]$ of $x$ and is of klt type.
\end{proposition}

\begin{proof}
By Proposition~\ref{prop:lift-iteration-Cox-ring}, we can lift the action of $\mathbb{T}$ to the iteration of aff-local Cox rings of $X$ at $x$. Analogous to the construction of the solvable group $S^{(n)}$ appearing in part (4) of Proposition~\ref{prop:Cox-ring-iteration} as an iterated extension of the lifted characteristic quasitori $H^{(n)}$,   we can in fact construct a solvable group $S_{\mathbb{T}}^{(n)}$ as an iterated extension of $\mathbb{T}$ by $H^{(1)}$, \ldots, $H^{(n)}$  that acts on $X^{(n)}$ in such a way that
\begin{equation}\label{eq:solvablerepresentation}
X^{(n)} /\!/S_{\mathbb{T}}^{(n)} \cong \bigl(\bigl(\bigl(\bigl(X^{(n)} /\!/H^{(n)}\bigl)/\!/ H^{(n-1)}\bigr) \cdots \bigr)/\!/H^{(1)} \bigr) /\!/\mathbb{T} \cong X /\!/\mathbb{T};
\end{equation}
see~\cite[Pf.~of Thm.~1.6]{ABHW18} and the diagram of Proposition~\ref{prop:lift-iteration-Cox-ring}. 
The identity component of $S_{\mathbb{T}}^{(n)}$ is an algebraic torus $\mathbb{T}'$ and the group of components is a finite solvable group $S_{\mathbb{T}}^{\mathrm{fin}}$. From  \eqref{eq:solvablerepresentation} we get an isomorphism
$$
X /\!/\mathbb{T} \cong  \bigl(X^{(n)} /\!/\mathbb{T}' \bigr)  /\!/ S_{\mathbb{T}}^{\mathrm{fin}}.
$$
Now, since $\mathcal{O}_{X^{(n)},x^{(n)}}$ is factorial, by Lemma~\ref{lem:shrinking-around-fixed-point}, without loss of generality we may assume that $X^{(n)}=X^{(n)}_g$ is locally factorial, and thus in particular canonical Gorenstein. As a consequence, we can apply Proposition~\ref{prop:torus-quotient-klt} in order to see that $X^{(n)} /\!/\mathbb{T}'$ is of klt type. Then, by Proposition~\ref{prop:finite-quotient}, also the finite quotient by $S_{\mathbb{T}}^{\mathrm{fin}}$ is of klt type and the claim follows.
\end{proof}

\begin{proof}
[Proof of Theorem~\ref{introtm:reductive-quot}]
Let $[x] \in X/\!/G$, where $x \in \pi^{-1}([x])$ is chosen such that the $G$-orbit $G.x$ is closed in $X$. Since the property of being of klt type is an \'etale local property by Proposition~\ref{prop:klt-type-etale}, it suffices to prove that there is an \'etale neighborhood $V \to  X/\!/G$ of $[x]$ of klt type. 
 As $G.x$ is closed, the stabilizer subgroup $G_x  \subseteq G$ is reductive \textcolor{black}{e.g.~by \cite[Prop.~in Sect.~2]{Lun73}}. Furthermore, Luna's \'etale slice theorem tells us that there is a $G_x$-invariant locally closed affine subvariety $W \subseteq X$ containing $x$ such that the natural map $\varphi: G \times_{G_x} W \to X$ is \'etale, maps onto a $\pi$-saturated open affine subspace of $X$ and is such that the induced map of quotients realizes $W/\!/G_x$ as an \'etale neighborhood $U \to X/\!/G$ of $[x]$, see e.g.~\cite[Sec.~1.1]{Dre04} or the original source \cite{Lun73}. As $\varphi$ is \'etale, the tube $G \times_{G_x} W$ is of klt type; the boundary divisor $\Delta$ can be chosen to be a pullback from $X$. The natural map $q: G\times_{G_x} W \to G/G_x$ realizes the tube as a fibre bundle over the (affine) orbit $G/G_x$ with typical fibre $W$ and structure group $G_x$. As it follows from \cite[Lem.~5.17]{KM98} that for a general point $p \in G/G_x$ the pair $(q^{-1}(p), \Delta|_{q^{-1}(p)})$ is klt, we may therefore assume from now on that $X$ is affine and that the reductive group $G$ fixes $x \in X$.

We denote by $G^0$ the connected component of the identity $e \in G$. The group $G^0$ is a normal connected reductive subgroup of $G$. We further denote by $F:=G/G^0$ the finite group of components of $G$. The derived subgroup $(G^0)^{\mathrm{ss}}$ is a connected semisimple group, and the quotient $T:=G^0/(G^0)^{\mathrm{ss}}$ is an algebraic torus. Note that since the derived subgroup $(G^0)^{\mathrm{ss}}$ is a \emph{characteristic subgroup} of $G^0$, the former is even normal in $G$, see e.g.~\cite[1.5.6.~(iii)]{Rob93}. By applying~\cite[Lem.~4.2.1.3]{ADHL15} two times (note that connectedness of the involved groups is not necessary for the proof provided there),  we have the following commutative diagram of quotient morphisms
$$
 \xymatrix{
    X \ar[r] \ar@/_1.6pc/[rrr]
    & X/\!/(G^0)^{\mathrm{ss}} \ar[r] 
    & \left(X/\!/(G^0)^{\mathrm{ss}}\right)/\!/T \ar[r]
    & \left(\left(X/\!/(G^0)^{\mathrm{ss}}\right)/\!/T \right)/\!/F = X/\!/G.\\ ~ 
    }\vspace{-0.5cm}
$$
Thus, after applying Proposition~\ref{prop:torus-quotient} (the torus case) and Proposition~\ref{prop:finite-quotient} (the case of finite groups) we are left to deal with the case that $G$ is semisimple and acts on an affine variety $X$ with a fixed point $x$. By Proposition~\ref{prop:Cox-ring-iteration}, we can choose $n \in \mathbb{N}$, such that the spectrum $X^{(n)}$ of the iterated aff-local Cox ring is factorial at the unique $H^{(n)}$-fixed point $x^{(n)} \in X^{(n)}$. By part (4) of Proposition~\ref{prop:lift-iteration-Cox-ring}, the action of the semisimple covering group $G^{(n)} \to G$ fixes $x^{(n)}$ as well. Let $y^{(n)}$ be the image of $x^{(n)}$ in $Y^{(n)}:=X^{(n)}/\!/G^{(n)}$. Then, by Lemma~\ref{le:factorial-semisimple-quotient} there is an affine open locally factorial neighborhood $y^{(n)} \in U^{(n)} \subseteq X^{(n)}/\!/G^{(n)}$ with canonical singularities. By the same considerations as in the proof of Lemma~\ref{le:factorial-semisimple-quotient}, we can assume that $U^{(n)}$ is a principal open affine of $Y^{(n)}$ and $H^{(n)}$-stable for the induced action of $H^{(n)}$ on $Y^{(n)}$. But then another application of Proposition~\ref{prop:torus-quotient} implies that $U^{(n-1)}:=U^{(n)}/\!/H^{(n)}$ is an open subset of klt type of $Y^{(n-1)}$ containing the image $y^{(n-1)}$ of $x^{(n-1)}$. Applying this procedure $n$ times, from Item (6) of Proposition~\ref{prop:lift-iteration-Cox-ring} we obtain an open subset of klt type $U^{(0)} \subseteq Y^{(0)}=X/\!/G$ containing the image $[x]$ of $x$. This concludes the proof of Theorem~\ref{introtm:reductive-quot}.
\end{proof}

\begin{proof}
[Proof of Corollary~\ref{introcor:reduct-quot}]
By definition, a reductive quotient singularity has an \'etale neighborhood that is isomorphic to an \'etale neighborhood of $[0] \in V/\!/G $ for some finite-dimensional complex representation $V$ of a reductive group $G$. The latter one has singularities of klt type by Theorem~\ref{introtm:reductive-quot}, as $V$ is smooth and hence of klt type. Proposition~\ref{prop:klt-type-etale} therefore yields the claim. 
\end{proof}

\begin{remark}{\em
Instead of appealing to Theorem~\ref{introtm:reductive-quot} we could also prove the main statement of Corollary~\ref{introcor:reduct-quot}, namely the fact that $V/\!/G$ is of klt type, avoiding iteration of Cox rings as follows. First, as in the proof of Theorem~\ref{introtm:reductive-quot} above, we write $V/\!/G \cong ((V/\!/(G^0)^{ss})/\!/\mathbb{T})/\!/F$. By the classical observation \cite[Thm.~II.3.17]{AGIV}, the affine variety $V/\!/(G^0)^{ss}$ is factorial, hence rational Gorenstein, and in particular canonical. Then, Propositions~\ref{prop:torus-quotient-klt} and \ref{prop:finite-quotient} yield the claim.}
\end{remark}

\begin{proof}
[Proof of Theorem~\ref{introthm:proj-quot}]
Let $\pi\colon X\rightarrow X/\!/G$ be the good quotient.
A good quotient is an affine morphism, so we can choose an affine covering $\{U_i\}_{i\in I}$ so that 
$\pi^{-1}(U_i) \rightarrow U_i$ are good quotients and
$\{\pi_1^{-1}(U_i)\}_{i\in I}$ give an affine covering of $X$.
Since $X$ is of klt type, 
then each of the $\pi^{-1}(U_i)$ is of klt type as well.
By Theorem~\ref{introcor:reduct-quot}, we conclude that $U_i$ is of klt type for each $i\in I$. Then, $X/\!/G$ is of klt type, as for quasi-projective varieties the property of being of klt type is Zariski-local (see the first paragraph of the proof of Proposition~\ref{prop:klt-type-etale}). 
\end{proof}

\begin{proof}[Proof of Corollary~\ref{introcor:GITquotients}]
Since GIT quotients are in particular good quotients with quasi-projective quotient space \cite[Thm.~1.10]{MFK94}, Theorem~\ref{introthm:proj-quot} applies to prove this claim.
\end{proof}

\section{Proofs of further applications of the main results}\label{sect:applicationsproof}

In this section, we give some further applications
of the main theorem to symplectic quotients, 
good moduli spaces, Mori Dream spaces, and collapsing of homogeneous bundles.

\subsection{Application to symplectic quotients}
In this subsection, we prove Theorem~\ref{introthm:Hamiltonian}, which is an application to symplectic quotients of K\"ahler manifolds.

\begin{proof}[Proof of Theorem~\ref{introthm:Hamiltonian}]
Let us quickly summarize the quotient theory for Hamiltonian actions on K\"ahler manifolds, which is analogous to Geometric Invariant Theory in the algebraic setup; e.g.~see \cite{Sja95, HH99}.  Let $X^{ss}_\mu = \{x \in X \mid \overline{G.x} \cap \mu^{-1}(0) \neq \emptyset\}$ be the open subset of $\mu$-semistable points in $X$. Then, the complex structure on the Hausdorff topological space $\mu^{-1}(0)/K$ is constructed in such a way that there is a holomorphic map $\pi: X^{ss}_\mu \to \mu^{-1}(0)/K$ with $(\pi_*\mathcal{O}_{X^{ss}_\mu})^G = \mathcal{O}_{\mu^{-1}(0)/K}$ having the property that every $\pi$-fibre contains a unique $G$-orbit that is closed in $X^{ss}_\mu$.

Let now $p \in \mu^{-1}(0)/K$ and let $x \in \pi^{-1}(p)$ have closed orbit. Then, the stabilizer $G_x$ is reductive, and there is a holomorphic slice at $x$. I.e., there exists an open Stein neighborhood $U$ of $p$ in $\mu^{-1}(0)/K$ such that $\pi^{-1}(U)$ is Stein and $G$-equivariantly biholomorphic to $G \times_{G_x} W$, where $W$ is an open, Stein, $G_x$-invariant subset of the (algebraic) $G$-representation $V:=T_xX/T_x(G.x)$ that is saturated with respect to the invariant-theoretic quotient morphism $V \to V /\!/G_x $; see \cite[(2.7) Theorem]{HL94} or \cite[Thm.~1.12]{Sja95}. It follows that $U$ is bihomolorphic to the Euclidean open image of $W$ in $V /\!/G_x$. Since the latter variety is affine and of klt type by Theorem~\ref{introtm:reductive-quot}, this concludes the proof. \end{proof}

\subsection{Applications to good moduli spaces}
In this subsection, we prove some applications to the theory of good moduli spaces.

\begin{proof}[Proof of Theorem~\ref{thm:K-moduli}]
\textcolor{black}{
Let $\overline{\mathcal{X}}_{n,v}$ be the Artin stack of $\mathbb{Q}$-Gorenstein smoothable K-semistable $\mathbb{Q}$-Fano varieties of dimension $n$ and volume $v$. As these form a bounded family, there exists a locally closed subscheme $\overline{P} \subset \mathrm{Hilb}_\chi(\mathbb{P}^N)$ invariant under the natural action of $G:=\mathrm{SL}_{N+1}$ such that $\overline{\mathcal{X}}_{n,v} \cong [\overline{P}/G]$ and such that the locus $P$ parametrizing smooth $K$-semistable Fano varieties is $G$-stable, open, and dense in $\overline{P}$, see \cite[Sect.~6.3]{Xu20}. The main result of \cite{LWX19} implies that the good quotient $\pi: \overline{\mathcal{X}}_{n,v} \to \overline{X}_{n,v} = \overline{P} /\!/G$ exists as a separated scheme of finite type; see also the discussion in \cite[Sect.~7.2]{Xu20}. The closed points of $\overline{X}_{n,v}$ parametrize the isomorphism classes of $n$-dimensional $\mathbb{Q}$-Gorenstein smoothable K-polystable $\mathbb{Q}$-Fano varieties of volume $v$; these correspond to the closed $G$-orbits in $\overline{P}$.  The moduli space $X_{n,v}$ of K-polystable smooth Fano varieties is the open subset of $\overline{P} /\!/G$ consisting of those points for which the corresponding closed orbit lies in $P$, i.e., where the K-polystable representative is a \emph{smooth} Fano variety. As smooth Fano varieties have unobstructed deformations by Nakano vanishing, Serre duality, and Kodaira-Spencer theory \cite[Thm.~5.6]{KodairaBook} and as K-semistability is an open condition e.g. by \cite[Thm.~6.8]{Xu20}, the preimage $P:=\pi^{-1}(X_{n,v})$ is smooth; moreover, clearly $X_{n,v} = P/\!/G$. From~\cite{LWX18}, discussed also in \cite[Sect.~8]{Xu20}, we know that ${X}_{n,v}$ is quasi-projective. Hence, we may apply  Theorem~\ref{introthm:proj-quot}  to each connected component of $P$ and to the corresponding connected component of $X_{n,v}$ to find an effective divisor $B_{n,v}$ on $X_{n,v}$ such that the pair $(X_{n,v},B_{n,v})$ has klt singularities. }
\end{proof}

\begin{remark}\label{rem:badboundary}{\em
We recall that even if the open locus parametrizing smooth K-polystable Fano varieties
has klt type singularities, the proper moduli space parametrizing K-polystable $\qq$-Fanos, e.g., see \cite{LXZ21}, may have worse singularities~\cite{MGS21, Pet21a, Pet21b}.}
\end{remark}

\begin{proof}[Proof of Theorem~\ref{introthm:stacksmoduli}]
Again, by Proposition~\ref{prop:klt-type-etale} it suffices to show that every $p \in X$ has an \'etale neighbourhood of klt type. For this, we proceed along the lines of \cite[Proof of Thm.~4.12]{AHR20}. If $x \in \mathcal{X}(k)$ is the unique point lying over $p$, see \cite[Prop.~9.1]{Alp13}, the stabilizer $G_x$ of $x$ is (linearly) reductive by \cite[Prop.~12.14]{Alp13}. Moreover, as $\mathcal{X}$ is smooth, \cite[Thm.~1.2]{AHR20} guarantees that there exists a smooth affine $G_x$-scheme $W$ containing a $G_x$-fixed point $w \in W$ together with an affine, étale morphism $f: ([W/G_x], w) \to (\mathcal{X},x)$ that induces an isomorphism of stabilizer groups at $w$. After shrinking, Luna’s fundamental lemma for stacks, \cite[Prop.~4.13]{AHR20} then yields an \'etale morphism of good moduli spaces $(W/\!/G_x, [w]) \to (X, p)$. By Theorem~\ref{introtm:reductive-quot}, the invariant-theoretic quotient $W/\!/G_x$ is of klt type; so, we are done.
\end{proof}

\subsection{Applications to Mori Dream Spaces}
In this subsection, we prove that certain good quotients of Mori Dream Spaces are again Mori Dream Spaces.
As a consequence, we prove that moduli spaces of quiver representations are of Fano type.

\textcolor{black}{
\begin{lemma}
\label{le:kltcoxrings}
    Let $X$ be a Mori Dream Space. Choose a finitely generated group $K\subseteq \WDiv(X)$ of Weil divisors, a subgroup $L \subseteq K$ yielding either $K/L \cong \Cl(X)$ or $K/L \cong \Cl(X)/\Cl(X)_{\mathrm{tors}}$, and a character $\chi\colon L \to \kk(X)^*$. If the Cox ring $R_{X,K,L,\chi}$ is of klt type, then also for any other choice of $K', L'$ and $\chi'$ with $K'/L' \cong \Cl(X)$ or $ K'/L' \cong \Cl(X)/\Cl(X)_{\mathrm{tors}}$, 
    the respective Cox ring $R_{X,K',L',\chi'}$ is of klt type.
\end{lemma}}
\textcolor{black}{\begin{proof}
Let $R_{X,K,L,\chi}$ be the Cox ring with respect to a finitely generated group $K\subseteq \WDiv(X)$ of Weil divisors, a subgroup $L \subseteq K$ yielding $K/L \cong \Cl(X)$ and a character $\chi\colon L \to \kk(X)^*$, cf.~Construction~\ref{constr:CoxSheaf}.  Let $K' \subseteq K$ be a subgroup mapping isomorphically to the free part $\Cl(X)/\Cl(X)_{\mathrm{tors}}$, and denote the corresponding ring of sections by $R_{X, K'}$. As a consequence of \cite[Pf.~of~Lem.~2.21]{GOST15}, the natural inclusion $R_{X,K'} \hookrightarrow R_{X,K,L,\chi}$ yields a finite Galois morphism that is \'etale in codimension one. Hence, using Proposition~\ref{prop:finite-quotient} and \cite[Prop.~5.20]{KM98}, we infer that one of these spectra is of klt type if and only if the other is so.  Now let $K'$ and $K''$ be different groups of Weil divisors mapping isomorphically to $\Cl(X)/\Cl(X)_{\mathrm{tors}}$.
For any $m \in \mathbb{Z}$, the natural inclusion $R_{X,mK'} \hookrightarrow R_{X,K'}$ yields a finite Galois morphism, \'etale in codimension one, as above. So the spectrum of $R_{X,mK'}$ is of klt type if and only if the spectrum of $R_{X,K'}$ is so. Thus, replacing $K'$ and $K''$ by appropriate multiples $mK'$ and $mK''$ respectively, we can assume that the respective generators $D'_i$ and $D''_i$ differ only by a principal divisor $\ddivv(f_i)$. Then~\cite[Pf.~of~Constr.~1.4.1.1]{ADHL15} shows that $R_{X,mK'}$ and $R_{X,mK''}$ are isomorphic. In summary, if one Cox ring is klt, all Cox rings are klt, as claimed.
\end{proof}
}

\begin{proof}[Proof of Theorem~\ref{introthm:MDSquotients}]

By Lemma~\ref{le:kltcoxrings}, we know that even if $X$ allows nontrivial global invertible functions and has hence several non-isomorphic Cox rings, all (spectra of) Cox rings ---no matter if taken with respect to the free part of the class group (as in~\cite{GOST15}) or with respect to the whole class group--- are klt if and only if one (spectrum of a) Cox ring is klt.

Next, we note that the proof of~\cite[Thm.~1.2]{Bae11} does not rely on the property of having only constant invertible global functions, so it yields that $\kk[U]$ and all Cox rings over $U$ are finitely generated. By~\cite[Lem.~3.1]{Bae11}, the class group $\Cl(U/\!/G)$ is finitely generated, and since $U/\!/G$ is normal, we may additionally assume that it is smooth when analyzing (the singularities of) its Cox ring(s). Consider a group $K$ of Weil divisors on $U/\!/G$ mapping isomorphically to the free part of $\Cl(U/\!/G)$. By the above considerations, it suffices to show that $R_{U/\!/G,K}$ is finitely generated and klt. The pullback group $K':=\pi^*K$ is a group of invariant Weil divisors on $U$. Thus, there is a canonical $G$-linearization of the corresponding divisorial ring $R_{U,K'}$, and the proof of~\cite[Thm.~1.1]{Bae11} shows that indeed
\[
(R_{U,K'})^G \cong R_{U/\!/G,K}.
\]
Thus, by a first application of Theorem~\ref{introtm:reductive-quot}, we are left with showing that $\Spec R_{U,K'}$ is klt. 

For this last step, we follow the proof of~\cite[Thm.~1.2]{Bae11} and let $K^0 \subseteq K'$ be the subgroup consisting of principal divisors. By the above considerations, up to a finite quasi-\'etale covering, we can assume that $K'=K^0 \oplus K^1$ splits as a direct sum, where $K^1$ is a subgroup of a group $K^2$ of Weil divisors mapping surjectively onto $\Cl(U)$. If $\mathrm{rk}(K^0)=r$, then~\cite[Rem.~1.3.1.4]{ADHL15} shows that
$$
R_{U,K'} \cong R_{U,K^1} \otimes \kk[x_1^{\pm 1},\ldots,x_r^{\pm 1}],
$$
so $\Spec R_{U,K'}$ is of klt type if and only if $\Spec R_{U,K^1}$ is so. The latter variety is a torus quotient of $\Spec R_{U,K^2}$, the spectrum of a "full" Cox ring. This Cox ring in turn is (at most) a localization of a Cox ring of $X$, compare~\cite[Pf.~of~Thm.~1.2]{Bae11}, so in particular, it is klt by our assumption. Therefore, Theorem~\ref{introtm:reductive-quot} applies a second time to show that the torus quotient $\Spec R_{U,K^1}$ is of klt type. This concludes the proof.
\end{proof}

\begin{proof}[Proof of Corollary~\ref{introcor:Fanoquotients}]
Since $X$ is proper and $L$ is ample, the GIT-quotient $X_L^{ss}/\!/G$ is projective, see, e.g.,~\cite[p.~40]{MFK94}. Thus, by Theorem~\ref{introthm:MDSquotients}, it is a projective MDS with klt Cox ring. So~\cite[Cor.~5.3]{GOST15} applies to prove the claim that $X_L^{ss}/\!/G$ is of Fano type.
\end{proof}

\begin{remark}\label{rem:ChowQuotients}{\em
It is important that we are considering GIT-quotients in the above discussion. For example, the ``normalized Chow quotient" of $(\mathbb{P}^1)^n$ by the natural $\mathrm{GL}_2$-action is isomorphic to the moduli space $\overline{M}_{0,n}$ of stable $n$-pointed curves of genus $0$. At the same time, this moduli space can also be realized as a normalized Chow quotient of the Grassmannian $\mathrm{Gr}(2, n)$ by a torus action, see \cite[Thm.~4.1.8]{Kap93}. However,  $\overline{M}_{0,n}$ is not of Fano type for $n\geq 7$.  Indeed, by \cite[Thm.~1.3(1) and Lem.~3.5]{KMcK13} the anti-canonical divisors of the manifolds in this range do not lie inside the interior of the respective effective cone; i.e., $-K$ is not big. In fact, it was shown in \cite{CT15} that $\overline{M}_{0,n}$ is not even a Mori Dream Space for $n \geq 134$ (a bound which has since been lowered to $n \geq 10$, see~\cite{HKL18}).}
\end{remark}

\begin{proof}[Proof of Corollary~\ref{introcor:quivers}] We use results summarized for example in \cite[Sect.~10.1]{Kir16}. If the quiver has vertex set $I=\{1, \dots, k\}$ and arrow set $A$, and for each arrow $a\in A$ we denote by $s(a)$ its source and by $t(a)$ its target, then we define the \emph{representation space} to be \[R := \bigoplus_{a \in A} \mathrm{Hom}_\mathbb{K}(\mathbb{K}^{n_{s(a)}}, \mathbb{K}^{n_{t(a)}}).\] The $\mathbb{K}$-vector space $R$ carries a natural $G:=\mathrm{GL}_{n_1} \times \cdots \times \mathrm{GL}_{n_k}$-representation by pre- and post-composition that factors through the projectivization $PG$. As there are no oriented cycles, the representation does not have any invariants and the invariant-theoretic quotient $R /\!/ PG$ (parametrizing semisimple representations of $\mathcal{Q}$) is a point. The character $\theta$ defines a linearization of the trivial line bundle over $R$ and the moduli space under discussion 
\[M^{ss}_\mathcal{Q}(\mathbf{n}, \theta) = R^{ss}_{\theta} /\!/ PG \] admits a natural projective morphism to $R /\!/ PG = \{pt.\}$ and is therefore projective.  Let $\mathbb{P} = \mathbb{P}(R \oplus \mathbb{K})$ be the projective completion of $R$, endowed with the $PG$-action induced by letting it act trivially on the $\mathbb{K}$-summand. Then, $R^{ss}_\theta$ is a $PG$-invariant open subset of $\mathbb{P}$ admitting a good quotient with projective quotient space. By \cite[Thm.~3.3]{Hau04}, there hence exists a $PG$-linearization in some power $H$ of the hyperplane bundle $\mathcal{O}(1)$ such that $R^{ss}_\theta = \mathbb{P}^{ss}_H$. We may therefore apply Corollary~\ref{introcor:Fanoquotients} to conclude that $M^{ss}_\mathcal{Q}(\mathbf{n}, \theta)$ = $\mathbb{P}^{ss}_H /\!/ PG$ is of Fano type, as claimed. 
\end{proof}

\subsection{Applications to collapsing of homogeneous bundles}
In this subsection, we apply our main result to study the singularities of collapsings of homogeneous bundles.

\begin{proof}[Proof of Thm.~\ref{introthm:Collapsing}]
We roughly follow the lines of the proof of~\cite[Thm. 3.10~(4)]{Lo20}, while staying in characteristic $0$ and at the same time being able to simplify some parts of the argument.
We denote by $q\colon G \times_{P} X \to W$ the collapsing map and by $L$ the reductive factor group in the Levi decomposition $P=L \ltimes U_P$. We may assume that $G$ is semisimple and simply connected by the considerations in~\cite[p.~12]{Lo20}; in particular, $\mathbb{K}[G]$ is then a factorial ring, see e.g.~\cite[Prop.~1]{Pop74}. Furthermore, as a consequence of \cite[Thm.~2 b)]{Ke76} we have $q_*(\mathcal{O}_{G \times_P X}) = \mathcal{O}_{G\cdot X}$. Therefore, as moreover $U_P$ acts trivially on $X$ by assumption, we obtain the following chain of isomorphisms
\begin{align*}\label{eq:chainofisos}
\mathbb{K}[G\cdot X] &\cong H^0\bigl(G \times_{P} X, \mathcal{O}_{G \times_{P} X}\bigr) \cong H^0\bigl((G \times_{U_P} X)/L, \mathcal{O}_{(G \times_{U_P} X)/L}\bigr)  \\ &\cong H^0\bigl((G/U_P \times X)/L, \mathcal{O}_{(G/U_P \times X)/L}\bigr) \cong (\mathbb{K}[G/U_P] \otimes \mathbb{K}[X])^L \\ &= (\mathbb{K}[G]^{U_P} \otimes \mathbb{K}[X])^L;
\end{align*}
here, one should note that $\mathbb{K}[G]^{U_P}$ is finitely generated by~\cite{Gro83}.
Theorem~\ref{introtm:reductive-quot} will imply the claim once we know that the variety $\mathrm{Spec}(\mathbb{K}[G]^{U_P})$ is of klt type.  However, as discussed in \cite[Lem.~3.7, Rem.~5.6]{Has12}, it is even factorial with rational singularities; see also \cite[Sect.~3, claim II]{Gro89}. In particular, $\mathbb{K}[G]^{U_P}$ has canonical Gorenstein singularities, so we are done.
\end{proof}

\begin{remark}{\em
\textcolor{black}{In the special case that $X$ equals the whole of $V$} and under the additional assumption that the reductive factor group $L$ in the Levi decomposition of $P$ is semisimple, in the above proof we may apply \cite[Thm.~II.3.17]{AGIV} to show that \textcolor{black}{$G \cdot X = G\cdot V$} has in fact canonical Gorenstein singularities.  }
\end{remark}

\bibliographystyle{habbrv}
\bibliography{BGLM_final}

\vspace{0.5cm}
\end{document}